\documentclass[10pt,a4paper]{amsart}

\usepackage{setspace}
\usepackage{amsfonts}
\usepackage{amsmath}
\usepackage{amsthm}
\usepackage{hyperref}

\usepackage{graphicx}
\usepackage{stmaryrd}
\usepackage{hyperref}

\usepackage{array}

\hypersetup{
  pdftitle   = {},
  pdfauthor  = {},
  pdfcreator = {\LaTeX\ with package \flqq hyperref\frqq}
}

\DeclareMathAlphabet{\mathpzc}{OT1}{pzc}{m}{it}

\newtheorem{theorem}{Theorem}[section]
\newtheorem*{theorem*}{Theorem}
\newtheorem{proposition}[theorem]{Proposition}
\newtheorem{lemma}[theorem]{Lemma}
\newtheorem*{lemma*}{Lemma}

\newtheorem*{conjecture*}{Conjecture}

\theoremstyle{definition}
\newtheorem{definition}[theorem]{Definition}
\newtheorem{example}[theorem]{Example}

\theoremstyle{remark}
\newtheorem{remark}[theorem]{Remark}

\newcommand{\disp}{\displaystyle}
\newcommand{\R}{\mathbb{R}}
\newcommand{\C}{\mathbb{C}}

\newcommand{\g}{\mathfrak{g}}
\newcommand{\h}{\mathfrak{h}}
\newcommand{\m}{\mathfrak{m}}
\newcommand{\mt}{\mathfrak{m}_\Theta}
\newcommand{\ul}{\mathfrak{u}}
\newcommand{\kt}{\mathfrak{k}_\Theta}

\newcommand{\su}{\mathfrak{su}}
\newcommand{\sll}{\mathfrak{sl}}
\newcommand{\Pp}{\mathcal{P}}
\newcommand{\Aal}{A_\alpha}
\newcommand{\Sal}{S_\alpha}

\newcommand{\Xtil}{\widetilde{X}}

\newcommand{\flag}{\mathbb{F}_\Theta}
\newcommand{\Pt}{P_{\Theta}}
\newcommand{\Kt}{K_{\Theta}}
\newcommand{\pimenos}{\Pi^+\setminus\langle\Theta\rangle}
\newcommand{\piteta}{\Pi(\Theta)}
\newcommand{\Lamjv}{\Lambda^{\#}}
\newcommand{\lamjv}{\lambda_\sigma^\#}
\newcommand{\lambjv}{\lambda^\#}
\newcommand{\svejv}{I_{\Lambda^\#}^\gamma(q)}
\newcommand{\sve}{I_{\Lambda}^\gamma(q)}
\newcommand{\sveNR}{I_{\Lambda_{NR}}^\gamma(q)}
\newcommand{\svee}{I_{\Lambda}^\gamma(q_0)}
\newcommand{\svejvv}{I_{\Lambda^\#}^\gamma(q_0)}

\newcommand{\dn}{\mathrm{d}}
\newcommand{\spp}{\mathfrak{sp}}
\newcommand{\Abeta}{A_{\beta}}
\newcommand{\Adelta}{A_{\delta}}

\DeclareMathOperator{\Ad}{Ad}
\DeclareMathOperator{\ad}{ad}
\DeclareMathOperator{\sen}{sen}
\DeclareMathOperator{\Sp}{span_\R}
\DeclareMathOperator{\Spp}{Sp}
\DeclareMathOperator{\Spc}{span_\C}
\DeclareMathOperator{\SU}{SU}
\DeclareMathOperator{\U}{U}
\DeclareMathOperator{\Id}{Id}

\usepackage{lineno}

\numberwithin{equation}{section}

\begin{document}


\title[Variational aspects of homogeneous geodesics on flag manifolds]{Variational aspects of homogeneous geodesics on generalized flag manifolds and applications}

\author{Rafaela F. do Prado and Lino Grama}
\address{Department of Mathematics - IMECC, University of Campinas - Brazil}
\email{linograma@gmail.com, rafaelafprado@gmail.com}
\thanks{LG is supported by Fapesp grant no. 2014/17337-0 and CNPq grant 476024/2012-9. RP is supported by CNPq grant 142259/2015-2.}

\begin{abstract}
We study conjugate points along homogeneous geodesics in generalized flag manifolds. This is done  by analyzing the second variation of the energy of such geodesics. We also give an example of how the homogeneous Ricci flow can evolve in such way to produce conjugate points in the complex projective space $\C P^{2n+1} = \text{Sp}(n+1)/(\text{U}(1)\times\text{Sp}(n))$.
\end{abstract}

\maketitle



\section*{Introduction}

In this paper we study global Riemannian geometry in a class of homogeneous spaces called generalized flag manifolds $U/K$ with $U$-invariant Riemannian metrics.



A generalized flag manifold is an homogeneous space $G/P$, where $G$ is a simple complex Lie group and $P$ is a parabolic subgroup of $G$. If $U$ is the compact real form of $G$ and $K=U\cap P$, then $G/P$ and $U/K$ are diffeomorphic and $K$ is the centralizer of a torus in $U$.

Using the invariance of the metric, many geometric objects can be described by their value at a single point, called the origin of the space (the trivial coset $o=[e]$ in $U/K$). Then, we can use a Lie theoretical approach to study curvatures and geodesics and, in many cases, one can translate a Riemannian geometry problem into a (non-trivial) algebraic problem. Some examples of this procedure can be found at \cite{A-A}, \cite{lino1}, \cite{itoh}, \cite{caiolinosan}, \cite{sm-negr}, \cite{ziller}.

In this work we study a special class of geodesics in homogeneous spaces, called {\em homogeneous geodesics}. A geodesic $\gamma$ is homogeneous if it is the orbit of a 1-parameter subgroup of $U$, that is, $$\gamma(t)=\exp tX\cdot o,$$ for some $X\in\mathfrak{u}=Lie(U)$.

An algebraic criteria to determine if a vector $X\in\mathfrak{u}$ gives rise to an homogeneous geodesic was given by Kowalski-Vanhecke in \cite{kow}. An homogeneous Riemannian manifold is called a {\em g.o. space} (geodesic orbit space) if every geodesic is homogeneous. Examples of g.o. spaces are the normal (cf. \cite{cheeger-ebin}) and the naturally reductive homogeneous spaces (cf. \cite{koba}). In \cite{A-A}, Arvanitoyergos and Alekseevsky classified  generalized flag manifolds that admits metrics (not homothetic to the normal metric) such that all geodesics are homogeneous.

In \cite{lino1} a new class of homogeneous geodesics were introduced, the {\em equigeodesics}, ones which are homogeneous curves that are geodesics with respect to any $U$-invariant metric.

From the global Riemannian geometry point of view, an important object is a Jacobi field along a geodesic. They provide information about minimizing properties of geodesics (conjugate points) and, therefore, about global geometry of a Riemannian manifold. The relation between conjugate points and calculus of variation is done by the Morse Index Theorem (cf. \cite{cheeger-ebin}). Some references regarding Jacobi fields in homogeneous spaces are the following papers by Chavel \cite{chavel}, Ziller \cite{ziller2} (in the naturally reducible case), and more recently Gonzalez-Davial and Navieira \cite{priconj} (for rank one normal homogeneous spaces).

In this work we study conjugate points along homogeneous geodesics in generalized flag manifolds equipped with an invariant metric (not necessarily a naturally reducible metric). We start by calculating the first and the second variations of the energy to such geodesics.

Now suppose that an homogeneous curve $\gamma$ is a geodesic in the flag manifold $U/K$ with respect to two invariant metrics $g_1$ and $g_2$ (we can suppose that $\gamma$ is an equigeodesic, for example). Our first main result gives a formula that allow us to compare the second variation of the energy of $\gamma$ with respect to $g_1$ and with respect to $g_2$ when the metric $g_2$ is obtained from a specific perturbation of the metric $g_1$ (such perturbation is called $\mathcal{P}$-perturbation). We have (see Section \ref{pertub} for details about the notation): 
\\

{\bf Theorem A: } {\em Let $U/K$ be a generalized flag manifold and let $\gamma: [0,a] \rightarrow U/K$ be an homogeneous curve and $g_2$ be an invariant metric obtained by a  $\Pp$-perturbation of the invariant metric $g_1$. Assume that  $\gamma$ is an homogeneous geodesic with respect to $g_1$ and $g_2$ having geodesic vector $X \in \mt$. Then 
$$I^{\gamma}_{g_2} = I^{\gamma}_{g_1} + 4\sum_{\alpha \in \sigma}\sum_{\sigma \in \piteta\setminus\Pp}\xi_\sigma \int_0^a(\tilde{f}_\alpha^2(t) + \tilde{g}_\alpha^2(t))dt,$$
where $I^{\gamma}_{g_i}, i=1,2 ,$ is the Morse index of $\gamma$ with respect to the metric $g_i$,     $\tilde{f}_\alpha(t)$ and $\tilde{g}_\alpha(t)$ are differentiable functions given by 
$$(\Ad{(\exp{-tX})q'(t))_{\mt}} = \sum_{\alpha \in \pimenos}(\tilde{f}_\alpha(t)\Aal + \tilde{g}_\alpha(t)\Sal),$$ and the real numbers $\xi_\sigma$ are determined by the $\mathcal{P}$-perturbation.}
\\

We remark that a similar result was proved in \cite{caiolinosan} in the context of harmonic maps from Riemann surfaces to homogeneous spaces.

We use Theorem A to study in details the case of the complex projective space $\mathbb{C} P^{2n+1}= \Spp(n+1) / \text{U}(1)\times \Spp(n)$. We remark that $\mathbb{C} P^{2n+1}$ viewed as a $\Spp(n+1)$-homogeneous space is {\em not} a Hermitian symmetric space. The (first) conjugate points of an homogeneous geodesic in $\mathbb{C} P^{2n+1}$ with respect to the normal metric ({\em standard}) are well known (see \cite{ziller2}, \cite{priconj}).

Let $\gamma$ be an equigeodesic in $(\mathbb{C} P^{2n+1}, g_{normal})$ and $t_0$ be a real number such that $\gamma(t_0)$ is the first conjugate point of $\gamma(0)$ along $\gamma$. As an application of Theorem A we deform the metric $g_{normal}$ in order to obtain for $t$ $(<t_0)$ $\gamma(0)$ and $\gamma(t)$ conjugate with respect to the deformed metric. More precisely,
\\

{\bf Theorem B: }{\em Consider the projective space $\C P^{2n+1} = \text{Sp}(n+1)/(\text{U}(1)\times\text{Sp}(n))$ equipped with $g_{normal}$, the normal invariant metric, and let $\gamma: \R \to \C P^{2n+1}$ be the equigeodesic 
\begin{equation*}
\gamma(t) = \exp{tA_{11}} \cdot o,
\end{equation*}
where $A_{11}\in \mathfrak{m}_2\subset\mathfrak{sp}(n+1)$. Fix $b\in \R$ such that $0 < b < \pi/2m$, where $m = \sqrt{2n+4}/(2n
+4)$. Then there exist an invariant metric $g$ obtained by a $\mathcal{P}$-pertubation of the metric $g_{normal}$ and $c \in \R$  $(0 < c \leq b)$, such that $\gamma(0)$ and $\gamma(c)$ are conjugate with respect to the metric $g$.}
\\

It is natural to ask if the deformation of the metric in Theorem B can be obtained by some geometric flow. Using the description of the Ricci flow for $Sp(n)$-invariant metrics in $\C P^{2n+1}$ obtained in \cite{gramamartins}, we have:
\\

{\bf Theorem C: }{\em Consider  $\C P^{2n+1} = \text{Sp}(n+1)/(\text{U}(1)\times\text{Sp}(n))$, with $n \geq 10$,  and $$b \in \left(\dfrac{8\sqrt{6}\pi^2}{3m(4n+3)},\dfrac{\pi}{2m}\right),$$ where $m = \sqrt{2n+4}/(2n
+4)$.
Let $\gamma(t)=\exp tX \cdot o$ be an equigeodesic with $X = A_{11} \in \m_1$.   Then the homogeneous Ricci flow $g(t)$ with $g(0)=g_{normal}$ evolves in such way that there exist $t_0\in (-\infty,0)$ and $c \in \R$ $(0 < c \leq b)$, such that $\gamma(0)$ and $\gamma(c)$ are conjugate with respect to the invariant metric $g(t_0)$.}
\\

Theorem C provides an example of how the Ricci flow can evolve in a way to produce conjugate points. There are several recent papers with important results about Ricci flow (and other geometric flows) in homogeneous spaces, for example, \cite{niko}, \cite{B-W}, \cite{lauret2}, \cite{lauret3} and references therein. 

In the final part of this paper we generalize Theorem B to arbitrary generalized flag manifolds. This is done using the Lie theoretical properties of flag manifolds. We introduce an algebraic concept called {\em perturbation pair with respect to a root $\alpha$}. Such perturbation pair consists of two roots of the Lie algebra of $U$ satisfying some properties (see Definition \ref{pair-pertub} for more details). In a short way, the root spaces associated to the perturbation pair provide directions to a $\Pp$-perturbation of the normal metric in order to produce conjugate points with respect to the perturbed metric. We obtain the following result:
\\

{\bf Theorem D:} {\em Let $U/K$ be a generalized flag manifold equipped with the normal (standard) metric $g_{normal}$. Let $\alpha \in \sigma_o \in \piteta$ and  consider the equigeodesic $\gamma: \R \to U/K$ given by $$\gamma(t) = \exp{tA_{\alpha}}\cdot o.$$ Suppose that there exists $a \in \R$ such that $\gamma(a)$ is the first conjugate point of $\gamma(0)$ along $\gamma$ with respect to the metric $g_{normal}$. If $\{\beta,\gamma\}$ is a perturbation pair with respect to $\alpha$, then there exist two real numbers $x,b$,  $0<x\leq b<a$ and an invariant metric $g$ obtained by a $\Pp$-perturbation of $g_{normal}$ such that $\gamma(x)$ is conjugate to $\gamma(0)$ with respect to the metric $g$.}
\\

As an application of Theorem D, we obtain a $\mathcal{P}$-perturbation of the normal metric in the Wallach flag manifold $SU(3)/T^2$ and produce conjugate points along an equigeodesic.

The paper is organized as follows: in section 1 we review the main facts about the geometry of flag manifolds; in section 2, we calculate the first and the second variation of the energy of a homogeneous geodesic; in section 3 we prove Theorem A; in section 4 we briefly review the $Sp(n)$-invariant geometry of $\C P^{2n+1}$ and the Ricci flow of invariant metrics and prove Theorems B and C. Finally, in section 5 we introduce the perturbation pair concept and prove Theorem D.

\section{Preliminaries: geometry of generalized flag manifolds}
\subsection{Generalized flag manifolds}
Let $\mathfrak{g}$ be a complex simple Lie algebra and take a Lie group $G$
with Lie algebra $\mathfrak{g}$. Given a Cartan subalgebra $\mathfrak{h}$ of 
$\mathfrak{g}$, denote by $\Pi $ the set of roots of the pair $\left( 
\mathfrak{g},\mathfrak{h}\right) $, so that

\begin{equation}
\mathfrak{g}=\mathfrak{h}\oplus \sum_{\alpha \in \Pi }\mathfrak{g}_{\alpha },
\end{equation}
where $\mathfrak{g}_{\alpha }=\{X\in \mathfrak{g};\,\forall H\in \mathfrak{h}%
,\,[H,X]=\alpha (H)X\}$ denotes the corresponding complex one-dimensional
root space.

We denote by $(\cdot ,\cdot) $ the Cartan-Killing form of $%
\mathfrak{g}$ and fix once and for all a Weyl basis of $\mathfrak{g}$ which
amounts to take $X_{\alpha }\in \mathfrak{g}_{\alpha }$ such that $(
X_{\alpha },X_{-\alpha }) =1$, and $[X_{\alpha },X_{\beta }]=m_{\alpha
,\beta }X_{\alpha +\beta }$ with $m_{\alpha ,\beta }\in \mathbb{R}$, $%
m_{-\alpha ,-\beta }=-m_{\alpha ,\beta }$ and $m_{\alpha ,\beta }=0$ if $%
\alpha +\beta $ is not a root (see \cite{hel}, \cite{sm-alg}).   

Recall that $( \cdot ,\cdot ) $ is nondegenerate on $\mathfrak{h}
$. Given $\alpha \in \mathfrak{h}^{*}$ we let $H_{\alpha }$ be given by $%
\alpha (\cdot )=( H_{\alpha },\cdot ) $, and denote by $%
\mathfrak{h}_{\R}$ the real subspace spanned by $H_{\alpha }$, $\alpha \in
\Pi $. Accordingly $\mathfrak{h}_{\mathbb{R}}^{*}$ stands for the real
subspace of the dual $\mathfrak{g}^{*}$ spanned by the roots.

Let $\Pi ^{+}$ be a choice of positive roots and $\Sigma $ the corresponding
set of simple roots. If $\Theta $ is a subset of $\Sigma $ we put $\langle
\Theta \rangle $ for the set of roots spanned by $\Theta $, and $\langle
\Theta \rangle ^{\pm }:=\langle \Theta \rangle \cap \Pi ^{\pm }$. We have 
\begin{equation}
\mathfrak{g}=\mathfrak{h}\oplus \sum_{\alpha \in {\langle \Theta \rangle }%
^{+}}\mathfrak{g}_{\alpha }\oplus \sum_{\alpha \in {\langle \Theta \rangle }%
^{+}}\mathfrak{g}_{-\alpha }\oplus \sum_{\beta \in \Pi ^{+}\setminus {%
\langle \Theta \rangle }}\mathfrak{g}_{\beta }\oplus \sum_{\beta \in \Pi
^{+}\setminus {\langle \Theta \rangle }}\mathfrak{g}_{-\beta }.
\end{equation}

Let 
\begin{equation}
\mathfrak{p}_{\Theta }=\mathfrak{h}\oplus \sum_{\alpha \in {\langle \Theta
\rangle }^{-}}\mathfrak{g}_{\alpha }\oplus \sum_{\alpha \in {\Pi }^{+}}%
\mathfrak{g}_{\alpha }
\end{equation}
be the standard parabolic subalgebra determined by $\Theta $. Put 
\begin{equation}
\mathfrak{q}_{\Theta }=\sum_{\beta \in \Pi ^{+}\setminus {\langle \Theta
\rangle }}\mathfrak{g}_{-\beta }
\end{equation}
so that $\mathfrak{g}=\mathfrak{q}_{\Theta }\oplus \mathfrak{p}_{\Theta }$.

The generalized flag manifold $\mathbb{F}_{\Theta }$ associated to $%
\mathfrak{p}_{\Theta }$ is defined as the homogeneous space 
\begin{equation}
\mathbb{F}_{\Theta }=G/P_{\Theta },
\end{equation}
where $P_{\Theta }$ is the normalizer of $\mathfrak{p}_{\Theta }$ in $G$.

We take as compact real form of $\mathfrak{g}$ the real subalgebra 
\begin{equation*}
\mathfrak{u}=\mathrm{span}_{\mathbb{R}}\{i\mathfrak{h}_{\mathbb{R}%
},A_{\alpha },iS_{\alpha }:\alpha \in \Pi \} 
\end{equation*}
where $A_{\alpha }=X_{\alpha }-X_{-\alpha }$ and $S_{\alpha }=X_{\alpha
}+X_{-\alpha }$.

Notice that a basis of $\ul$ is
\begin{equation}
\{\Aal, \Sal, iH_{\beta}: \, \alpha \in \Pi^+, \, \beta \in \Sigma\}.
\label{baseu}
\end{equation}

Denote by $U=\exp \mathfrak{u}$ the corresponding compact
real form of $G$ and write $K_{\Theta }=P_{\Theta }\cap U$. It is well known
that $U$ acts transitively on each $\mathbb{F}_{\Theta }$, that identifies
with $U/K_{\Theta }$.

Let $\mathfrak{k}_{\Theta }$ be the Lie algebra of $K_{\Theta }$ 
We have $%
\mathfrak{k}_{\Theta }=\mathfrak{u}\cap \mathfrak{p}_{\Theta }$

Denote by $o=eK_{\Theta}$ the origin of $\mathbb{F}_{\Theta }$. The tangent
space $T_{o}\mathbb{F}_{\Theta }$ can be identified with the orthogonal
complement of $\mathfrak{k}_{\Theta }$ in $\mathfrak{u}$, namely 
\begin{equation*}
T_{o}\mathbb{F}_{\Theta }= \mathfrak{m}_{\Theta }=\mathrm{span}_{\mathbb{R}%
}\{A_{\alpha },iS_{\alpha }:\alpha \notin \langle \Theta \rangle
\}=\sum_{\alpha \in \Pi \setminus \langle \Theta \rangle }\mathfrak{u}%
_{\alpha }, 
\end{equation*}
where $\mathfrak{u}_{\alpha }=\left( \mathfrak{g}_{\alpha }\oplus \mathfrak{g%
}_{-\alpha }\right) \cap \mathfrak{u}=\mathrm{span}_{\mathbb{R}}\{A_{\alpha
},iS_{\alpha }\}$.

The next lemma provides the Lie bracket between elements of $\ul$.
\begin{lemma}
\label{colchetes}
The Lie bracket between the elements of the base (\ref{baseu}) of $\ul$ is given by
\\

\begin{minipage}[t]{.5\textwidth}
$[\sqrt{-1}H_{\alpha},\Abeta] = \beta(H_{\alpha})S_{\beta}$ \\
$[\sqrt{-1}H_{\alpha},S_{\beta}] = -\beta(H_{\alpha})\Abeta$ \\
$[\Aal,\Sal] = 2\sqrt{-1}H_{\alpha}$
\end{minipage}
\begin{minipage}[t]{.5\textwidth}
$[\Aal,\Abeta] = m_{\alpha,\beta}A_{\alpha+\beta} + m_{-\alpha,\beta}A_{\alpha - \beta}$ \\
$[\Sal,S_{\beta}] = -m_{\alpha,\beta}A_{\alpha+\beta} - m_{\alpha,-\beta}A_{\alpha,-\beta}$ \\
$[\Aal,S_{\beta}] = m_{\alpha,\beta}S_{\alpha+\beta} + m_{\alpha,-\beta}S_{\alpha-\beta}$.
\end{minipage}

\end{lemma}

\subsection{Isotropy representation and invariant metrics}

Let $\flag = U/\Kt$ be a generalized flag manifold. The decomposition $\ul = \mt \oplus \kt$ satisfies $\Ad(\Kt)\mt \subset \mt$ and thus it is a reductive decomposition of $\ul$. Also, since the tangent space $T_o\flag$ at the origin $o=e\Kt$ is identified with $\mt$, the isotropy representation of $\Kt$ is identified with the restriction $\Ad(\Kt)|_{\mt}$ of the adjoint representation of $\Kt$ on $\ul$ to $\mt$. This representation is completely reductive and we can decompose $\mt$ as
$$\mt = \m_1 \oplus \cdots \oplus \m_n,$$
where each $\m_i$ is an irreducible component of the isotropy representation and can be written as a sum of $\ul_{\alpha}$'s (see \cite{itoh} for more details).

Let $\sigma_i \subset \pimenos$ be the subset of roots such that
$$\m_i = \sum_{\alpha \in \sigma_i}\ul_{\alpha}.$$
We will denote the set of $\sigma_i$'s by $\piteta$.

\begin{example}
Let $U/T$ be a maximal flag manifold ($\Theta = \emptyset$). In this case, each $\ul_\alpha$ is an irreducible component of the isotropy representation, where $\alpha \in \Pi^+$. We can suppose, without loss of generality, that $\Pi^+=\{\alpha_1,...,\alpha_n\}$. Thus, $\sigma_i = \{\alpha_i\}$ and $\piteta$ can be identified with $\Pi^+$. \qed
\end{example}

\begin{example} Let us describe the isotropy representation of the complex grassmannian (viewed as a generalized flag manifold) $$G_k \mathbb{C}^{n+k}=\SU(n+k)/\text{S}(\U(n)\times \U(k)).$$
The complexification of $\su(n+k)$ is the Lie algebra $\sll(n+k,\C)$. The root space decomposition of $\sll(n+k,\C)$ is given as follows. Consider the Cartan subalgebra $\h$ given by diagonal matrices of trace zero. Then, the root system of $\sll(n+k,\C)$ relative to $\h$ is composed by $\alpha_{ij} := \varepsilon_i - \varepsilon_j$, $1 \leq i \neq j \leq n+k$, where $\varepsilon_i$ is the functional given by $\varepsilon_i: diag\{a_1,\ldots,a_{n+k}\} \to a_i$.

A simple system of roots is
$$\Sigma = \{\alpha_{i,i+1}: \, 1 \leq i \leq n+k-1\}$$
and the set of positive roots relative to this simple system is
$$\Pi^+ = \{\alpha_{ij}: \, 1 \leq i < j \leq n+k \}.$$

Let $E_{ij}$ be the $(n+k)\times (n+k)$ matrix with 1 in the $ij$-position and 0 otherwise. Then, the root space $\g_{\alpha_{ij}}$ is the subspace generated by $E_{ij}$ over $\C$.

In the case of $\SU(n+k)/\text{S}(\U(n)\times\U(k))$, we have
$$\Theta = \{ \alpha_{i,i+1}: \, 1 \leq i \leq n+k-1, \, i \neq n\}$$
and it is an isotropically irreducible homogeneous space, that is, $\mt$ is itself an irreducible component of the isotropy representation. Since
$$\mt = \sum_{\substack{1\leq i\leq n \\ 1 \leq j \leq k}} \ul_{\alpha_{i,n+j}}$$
we have that $\sigma = \{\alpha_{i,n+j}: \, 1 \leq i \leq n, 1\leq j \leq k\}$. \qed

\end{example}

\begin{example} Consider the flag manifold $\C P^{5}=\Spp(3)/(\U(1)\times\Spp(2))$. 

The complexification of $\spp(3)$ is the Lie algebra $\spp(6,\C)$. The root space decomposition of $\spp(6,\C)$ is given as follows. Consider the Cartan subalgebra $\h$ of diagonal matrices. Then, the root system of $\spp(6,\C)$ relative to $\h$ is composed by $\alpha_{ij} := \varepsilon_i - \varepsilon_j$, $\alpha_{ij}^+:=\varepsilon_i + \varepsilon_j$, $\alpha_{ii} = 2\varepsilon_i$ and $-\alpha_{ii}$, $1 \leq i \neq j \leq 3$, where $\varepsilon_i$ is the functional given by $\varepsilon_i: diag\{a_1,a_2,a_{3}\} \to a_i$.

A simple system of roots is
$$\Sigma = \{\alpha_{12},\alpha_{23}, \alpha_{33}\}$$
and the set of positive roots relative to this simple system is
$$\Pi^+ = \{\alpha_{12},\alpha_{12}^+,\alpha_{13},\alpha_{13}^+,\alpha_{23},\alpha_{23}^+,\alpha_{11},\alpha_{22},\alpha_{33}\}.$$

Let $E_{ij}$ be the $6\times6$ matrix with 1 in the $ij$-position and 0 otherwise. Then, the root spaces are given by
\begin{eqnarray*}
\g_{\alpha_{ij}} & = & \Spc\{E_{ij} - E_{3+j,3+i}: \, 1\leq i \neq j \leq 3\}, \\
\g_{\alpha_{ij}^+} & = & \Spc\{E_{i,3+j} + E_{j,3+i}: \, 1 \leq i \neq j \leq 3\}, \\
\g_{-\alpha{ij}^+} & = & \Spc\{E_{3+i,j}+E_{3+j,i}: \, 1 \leq i \neq j \leq 3\}, \\
\g_{\alpha_{ii}} & = & \Spc\{E_{i,3+i}: \, 1 \leq i \leq 3\}, \\
\g_{-\alpha{ii}} & = & \Spc\{E_{3+i,i}: \, 1 \leq i \leq 3\}.
\end{eqnarray*}

In the case of $\C P^{5}=\Spp(3)/(U(1)\times\Spp(2))$, we have
$$\Theta = \{ \alpha_{23},\alpha_{33}\}$$
and $\mt = \m_1 \oplus \m_2$, where
$$\m_1 = \ul_{\alpha_{12}}\oplus\ul_{\alpha_{12}^+}\oplus \ul_{\alpha_{13}}\oplus\ul_{\alpha_{13}^+},$$
$$\m_2 = \ul_{\alpha_{11}}.$$

Then, $\sigma_1 = \{\alpha_{12},\alpha_{12}^+,\alpha_{13},\alpha_{13}^+\}$ and $\sigma_2 = \{\alpha_{11}\}$. \qed

\end{example}

There is a 1-1 correspondence between $U$-invariant metrics $g$ on $\flag$ and $\Ad(\Kt)$-invariant scalar products $B$ on $\mt$ (see for instance \cite{koba}). Any $B$ can be written as $$B(X,Y) = -(\Lambda X,Y),$$ with $X,Y \in \mt$, where $\Lambda$ is an $\Ad(\Kt)$-invariant positive symmetric operator on $\mt$ with respect to the Cartan-Killing form.

Notice that, as a consequence of Schur's Lemma, $$\Lambda|_{\m_i} = \lambda_i \Id|_{\m_i},$$with $\lambda_i>0$ for each $i=1,\cdots,n$ and hence any $\Ad(\Kt)$-invariant scalar product $B$ on $\mt$ can be described by
$$B(\cdot,\cdot) = -\lambda_1(\cdot,\cdot)|_{\m_1} \oplus \cdots \oplus -\lambda_n(\cdot,\cdot)|_{\m_n}.$$

\begin{remark}
In the next sections we abuse of the notation and denote an invariant metric $g$ on $\flag$ just by $\Lambda= (\lambda_\alpha)_{\alpha \in \pimenos}$, that is, a $n$-uple of positive real numbers indexed by the irreducible components of $\mathfrak{m}_\Theta$.
\end{remark}

\subsection{Homogeneous geodesics and equigeodescis} 
\begin{definition}
\label{defihomo} Let $(G/H,g)$ be a homogeneous Riemannian
manifold. A geodesic $\gamma(t)$ on $M$ through the origin $o$ is
called \textit{homogeneous} if it is the orbit of a $1$-parameter
subgroup of $G,$ that is, $$ \gamma(t)=(\exp tX)\cdot o, $$ where
$X \in \mathfrak{g}$. The vector $X$ is called a geodesic vector.
\end{definition}

Definition \ref{defihomo} establishes a 1:1 correspondence between
geodesic vectors $X$ and homogeneous geodesics at the origin. A very useful algebraic characterization of geodesic vectors was provided by Kowalski and Vanhecke. Let us recall this characterization.   
\begin{theorem}[\cite{kow}]\label{hecke}
Let $G/H$ be a homogeneous space and $g$ be a $G$-invariant metric determined by a scalar product $B$ at the origin $o=eH$. Then a vector $X \in
\mathfrak{g}\setminus\{ 0 \}$ is a geodesic vector if, and only if,
\begin{equation}
\label{eqnfund} B(X_{\mathfrak{m}},[X,Z]_{\mathfrak{m}})=0,
\end{equation}
for every $Z\in \mathfrak{m}$, where $\mathfrak{m}$ is identified with the tangent space at the origin $o$.
\end{theorem}
We also have the following result:
\begin{theorem}[\cite{kosz}]
\label{givevector} If $G$ is semi-simple then $G/H$ admits at least
$m=dim(M)$ mutually orthogonal homogenous geodesics through the
origin $o$.
\end{theorem}

\begin{definition}
A curve $\gamma$ on $G/H$ is an {\em equigeodesic} if it is a geodesic for
each invariant metrics on $G/H$. If the equigeodesic is of the
form $\gamma(t)=(\exp tX)\cdot o$,  where $X \in \mathfrak{g}$, we say that
$\gamma$ is a {\em homogeneous equigeodesic} and the vector $X$ is an
equigeodesic vector.
\end{definition}

If $U/K$ is a generalized flag manifold one can provide an algebraic characterization of equigeodesic vectors. Remember the reductive decomposition of $\mathfrak{u}=\mathfrak{k}\oplus \mathfrak{m}$ and the isomorphism of the tangent space at the origin of $U/K$ with the subspace $\mathfrak{m}$ of $\mathfrak{u}$.  We also identify invariant metrics $g$ on $U/K$ with $Ad(K)$-invariant scalar products $\Lambda$ on $\mathfrak{m}$.
\begin{proposition}[\cite{lino1}]
\label{equigeod} 
Let $U/K$ be a generalized flag manifold. A vector $X\in \mathfrak{m}$ is an equigeodesic vector
if, and only if,
\begin{equation}
\label{eqnlegal} [X,\Lambda X]_{\mathfrak{m}}=0,
\end{equation}
for each invariant metric $\Lambda$.
\end{proposition}

\begin{example}
Let $\flag = U/\Kt$ be a generalized flag manifold and let $\m_i$, $i=1,\ldots,n$, be the irreducible components of the isotropy representation. If $X \in \m_i$, then equation \ref{eqnlegal} is easily verified. A vector $X \in \m_i$ is called {\em trivial equigeodesic vector}.
\end{example}


\section{First and second variation of energy for homogeneous geodesics.}

%
Let  $\flag = G/\Pt = U/\Kt$ be a generalized flag manifold with reducible decomposition $\ul = \kt \oplus \mt$, where $\mt = \Sp{\{\Aal,\Sal; \alpha \in \pimenos\}}$, $\Lambda = (\lambda_\alpha)_{\alpha \in \pimenos}$  is an invariant metric on $\flag$ and  $\gamma: [0,a] \to \flag$ is differentiable curve given by
\begin{equation}
 \gamma(t) =  \exp{tX} \cdot o,
\end{equation} 
 $X \in \mt$. 

We will denote by $\widetilde{A}$ the vector field on $\flag$ defined by 
$$\widetilde{A}(x) = \dfrac{\mathrm{d}}{\mathrm{d}t} \exp{tA} \cdot x|_{t = 0} = (d\phi_x)_e(A),$$
$x \in \flag$. 

If $X = \disp\sum_{\alpha \in \pimenos}(a_\alpha \Aal + b_\alpha \Sal)$, $a_\alpha, \ b_\alpha \in \R$, hence the {\em energy} of $\gamma$ is given by 
\begin{eqnarray}
E(\gamma) & = & \int_0^a \! |\gamma'(t)|^2_\Lambda \, \dn t \nonumber \\
& = & \int_0^a \! \langle \Xtil(\gamma(t)), \Xtil(\gamma(t)) \rangle_\Lambda \, \dn t\nonumber \\
& = & \int_0^a \! \langle (d \phi_{\exp{tX}})_o(\Xtil(o)), (d \phi_{\exp{tX}})_o(\Xtil(o)) \rangle_\Lambda \,\dn t \nonumber \\
& = & \int_0^a \! \langle \Xtil(o), \Xtil(o) \rangle_\Lambda \, \dn t\nonumber \\
& = & \int_0^a \! B_\Lambda(X,X) \, \dn t\nonumber \\
& = & -a(\Lambda X, X)\nonumber \\
& = & -a\sum_{\alpha \in \pimenos}\lambda_\alpha(a_\alpha^2(\Aal,\Aal)+b_\alpha^2(\Sal,\Sal)) \nonumber \\
& = & 2a\sum_{\alpha \in \pimenos} \lambda_\alpha(a_\alpha^2+b_\alpha^2),\nonumber
\end{eqnarray}
where $B_\Lambda$ is the scalar product on $\mt$ associated to the invariant metric $\Lambda$.

Given a differentiable map $q: [0,a] \to \ul$, we define  a differentiable variation of the curve $\gamma$, $f: (-\varepsilon,\varepsilon) \times [0,a] \to \flag$ by $$f(s,t) = \exp{sq(t)} \cdot \gamma(t).$$ 

This variation is {\em proper} if, and only if, $q(0)$ and $\Ad{(\exp{-aX})}q(a) \in \kt$. Moreover, the variational vector field of  $f$ is 
$$V(t) = \dfrac{\partial f}{\partial s}(0,t) = (d\phi_{\gamma(t)})_e(q(t)).$$

The energy of the variation of $\gamma$ given by $f$ will be denoted by $E(s)$. Our first result is provide an explicit formula to $E(s)$, in therms of the invariant geometry of $\flag$. We will start proving some auxiliary lemmas.


\begin{lemma} \label{difexp} \normalfont{\textbf{(\cite{hel}, p. 95)}} Let $X \in \ul$. The differential of the exponential map from $\ul$ to $U$ at $X$ is given by 
$$d\exp_X = (dL_{\exp X})_e \circ T_X,$$
where
$$T_X = \sum_{n \geq 0} \dfrac{1}{(n+1)!}\ad(X)^n = \dfrac{e^{\ad(X)} - I}{\ad(X)}.$$

\end{lemma}

\begin{lemma}\label{dercurtrans} Let $\gamma: [0,a] \to \flag$ be a differentiable curve given by $\gamma(t) = \exp tX \cdot o$, $X \in \mt$. Consider a differentiable map $q: [0,a] \to \ul$ and $f: (-\varepsilon,\varepsilon) \times [0,a]$ be a variation of $\gamma$ defined by $f(s,t) = \exp sq(t) \cdot \gamma(t)$. Then,

$$\dfrac{\partial f}{\partial t}(s,t) = (d\phi_{\exp{sq(t)}})_{\gamma(t)}(\widetilde{C}_s(\gamma(t)) + \gamma'(t)),$$
where $C_s = T_{sq(t)}(sq'(t))$.
\end{lemma}
\begin{proof}
 We have
$$\dfrac{\partial f}{\partial t}(s,t) = (d\phi_{\gamma(t)})_{\exp{sq(t)}}(d \exp)_{sq(t)}(sq'(t)) + (d\phi_{\exp{sq(t)}})_{\gamma(t)}(\gamma'(t)).$$

By Lemma \ref{difexp},
\begin{eqnarray}
(d\phi_{\gamma(t)})_{{\exp{sq(t)}}}(d\exp)_{sq(t)}(sq'(t)) & = & (d\phi_{\gamma(t)})_{\exp{sq(t)}}(dL_{\exp{sq(t)}})_e(T_{sq(t)}(sq'(t))) \nonumber \\
            & = & d(\phi_{\gamma(t)} \circ L_{\exp{sq(t)}})_e(C_s) \nonumber \\
            & = & \dfrac{d}{dr} \ \phi_{\gamma(t)} \circ L_{\exp{sq(t)}} (\exp{rC_s}) \vert_{r=0} \nonumber \\
            & = & \dfrac{d}{dr} \ \phi_{\gamma(t)} (\exp{sq(t)} \exp{rC_s}) \vert_{r=0} \nonumber \\
            & = & \dfrac{d}{dr} \exp{sq(t)} \exp{rC_s} \cdot \gamma(t) \vert_{r=0} \nonumber \\
            & = & \dfrac{d}{dr} \phi_{\exp{sq(t)}}(\exp{rC_s} \cdot \gamma(t)) \vert_{r=0} \nonumber \\
            & = & (d\phi_{\exp{sq(t)}})_{\gamma(t)}(da_{\gamma(t)})_e(C_s) \nonumber \\
            & = & (d\phi_{\exp{sq(t)}})_{\gamma(t)}(\widetilde{C}_s(\gamma(t))). \nonumber
\end{eqnarray}
Therefore,
$$\dfrac{\partial f}{\partial t}(s,t) = (d\phi_{\exp{sq(t)}})_{\gamma(t)}(\widetilde{C}_s(\gamma(t)) + \gamma'(t)).$$
\end{proof}

\begin{lemma}\label{novaabor}
Let $\gamma: [0,a] \to \flag$ be a differentiable curve given by $\gamma(t) = \exp tX \cdot o$, $X \in \mt$, and $A \in \ul$. Then
$$\widetilde{A}(\gamma(t)) = (d \phi_{\exp{tX}})_o(\widetilde{\Ad{(\exp{-tX})}A}(o)).$$
\end{lemma}
\begin{proof}
We have
\begin{eqnarray*}
\widetilde{A}(\gamma(t)) & = & \dfrac{\dn}{\dn s} \exp{sA} \cdot \gamma(t)|_{s=0} \\
& = & \left. \dfrac{\dn}{\dn s} (\exp{tX} \exp{-tX}\exp{sA}\exp{tX}) \cdot o \, \right\vert_{s=0} \\
& = & \left. \dfrac{\dn}{\dn s} \phi_{\exp{tX}}(\exp{(s\Ad{(exp \, -tX)A})}\cdot o) \, \right\vert_{s=0}\\
& = & (d \phi_{\exp{tX}})_o(\widetilde{\Ad(\exp{-tX})A}(o)).
\end{eqnarray*}
\end{proof}

We are now able to compute the energy  $E(s)$ of a variation of $\gamma$, as well the first and second variation of the energy of homogeneous geodesics on $\flag$.


\begin{proposition}
Let $\gamma: [0,a] \to \flag$ be a differentiable curve defined by $\gamma(t) = \exp tX$, $X \in \mt$. Given a diffentiable map $q: [0,a] \to \ul$, consider the variation $f: (-\varepsilon,\varepsilon) \times [0,a]$ of $\gamma$ defined by $f(s,t) = \exp sq(t) \cdot \gamma(t)$. If $E: (-\varepsilon,\varepsilon) \to \R$ is the energy of $f$, then 
$$E(s) = E(\gamma) + \int_0^a \! B_\Lambda((\Ad(\exp{-tX})B_s)_{\mt},(\Ad(\exp{-tX})B_s)_{\mt} + 2X) \, \dn t,$$
where $B_\Lambda$ is the scalar product on $\mt$ associated to the invariant metric $\Lambda$.
\end{proposition}
\begin{proof} 
We have
\begin{eqnarray}
E(s) & = & \int_0^a \left| \dfrac{\partial f}{\partial t}(s,t) \right|^2_\Lambda dt \nonumber \\
& = & \int_0^a \vert (d\phi_{\exp{sq(t)}})_{\gamma(t)}(\widetilde{C}_s(\gamma(t)) + \gamma'(t)) \vert^2_\Lambda dt  \label{linha2} \\
                                 & = & \int_0^a \vert \widetilde{C}_s(\gamma(t)) + \gamma'(t) \vert^2_\Lambda dt \nonumber \\
                                 & = & E(\gamma) + \int_0^a \vert \widetilde{C}_s(\gamma(t)) \vert^2_\Lambda dt + 2\int_0^a \langle \widetilde{C}_s(\gamma(t)), \gamma'(t) \rangle_\Lambda dt \nonumber \\
                                 & = & E(\gamma) + \int_0^a \langle \widetilde{C}_s(\gamma(t)), \widetilde{C}_s(\gamma(t)) + 2\Xtil(\gamma(t)) \rangle_\Lambda dt \nonumber \\
                                 & = & E(\gamma) + \int_0^a \! B_\Lambda((\Ad(\exp{-tX})C_s)_{\mt},(\Ad(\exp{-tX})C_s)_{\mt} + 2X) \, \dn t , \label{ult-linha}
\end{eqnarray}
where $(\ref{linha2})$ follows by Lemma \ref{dercurtrans} and $(\ref{ult-linha})$ follows by Lemma \ref{novaabor}.
\end{proof}

\begin{proposition}
The {\em first variation of the energy} of $\gamma$ is given by 
$$E'(0) = 2\int_0^a B_\Lambda((\Ad{(\exp{-tX})}q'(t))_{\mt},X),$$
where $B_\Lambda$ is the scalar product on $\mt$ associated to the invariant metric $\Lambda$.
\label{privar}
\end{proposition}
\begin{proof} 
We have
$$E'(0) = 2\int_0^a \! B_\Lambda\left. \left(\dfrac{\partial}{\partial s}(\Ad(\exp{-tX})C_s)_{\mt}\right\vert_{s=0},X\right).$$

But
\begin{eqnarray*}
\left. \dfrac{\partial}{\partial s}\Ad{(\exp{-tX})}C_s \right\vert_{s=0} & = & \left.\dfrac{\partial}{\partial s}\exp{\ad(tX)}C_s\right\vert_{s=0}\\
& = & \left.\dfrac{\partial}{\partial s}\sum_{n \geq 0}\dfrac{t^n\ad(X)^n}{n!}C_s\right\vert_{s=0}\\
& = & \left.\dfrac{\partial}{\partial s}\sum_{n \geq 0}\dfrac{t^n\ad(X)^n}{n!}\sum_{k \geq 0}\dfrac{s^{k+1}\ad(q(t))^n}{(k+1)!}(q'(t))\right\vert_{s=0}\\
& = & \left.\dfrac{\partial}{\partial s}\left(\sum_{n \geq 0}\dfrac{t^n\ad(X)^n}{n!}sq'(t) + O(s^2)\right)\right\vert_{s=0}\\
& = & \sum_{n \geq 0}\dfrac{t^n\ad(X)^n}{n!}q'(t)\\
& = & \Ad{(\exp{-tX})}q'(t).
\end{eqnarray*}
Notice that we are using the expression $\Ad(\exp -tX)C_s = \exp\ad(tX)C_s$ because we are considering the {\em left} action of  $U$ on $\flag$ therefore the Lie algebra $\ul$ is the Lie algebra of {\em right} invariants vectors field on $U$.

Finally, we derive
$$E'(0) = 2 \int_0^a \! B_\Lambda((\Ad(\exp{-tX})q'(t))_{\mt},X).$$
\end{proof}
We know that geodesics are critical points of the energy. Hence, if the curve $\gamma$ is a geodesic with respect to the metric $\Lambda$ and $f$ is a proper variation then $E'(0)=0$. 

We will verify that homogeneous geodesics are in fact critical points of the energy. The main algebraic technique used to verify this fact will be useful later and we write it separately in the next Lemma. 


\begin{lemma} \label{vetgeo}
Let $A \in \ul$ and let $X \in \ul$ be a geodesic vector with respect to $\Lambda$. Then
$$B_\Lambda((\Ad(\exp -tX)A)_{\mt},X_{\mt}) = B_\Lambda(A_{\mt},X_{\mt}).$$
\end{lemma}

\begin{proof}
We have
\begin{eqnarray*}
(\Ad(\exp -tX)A)_{\mt} & = & (\exp t\ad(X)A)_{\mt} \\
& = & \left(\sum_{n \geq 0}\dfrac{t^n\ad(X)^n}{n!}A\right)_{\mt} \\
& = & A_{\mt} + \left[X,\sum_{n \geq 1}\dfrac{t^n\ad(X)^{n-1}}{n!}A\right]_{\mt}.
\end{eqnarray*}
Since $X$ is a geodesic vector with respect to $\Lambda$, $B_\Lambda(X_{\mt},[X,Y]_{\mt})_{\mt}=0$ for every $Y \in \ul$ (conf. Theorem \normalfont{\ref{hecke}}), and 
$$B_\Lambda\left(\left[X,\sum_{n \geq 1}\dfrac{t^n\ad(X)^{n-1}}{n!}A\right]_{\mt},X_{\mt}\right)=0.$$

Therefore,
$$B_\Lambda((\Ad(\exp -tX)A)_{\mt},X_{\mt}) = B_\Lambda(A_{\mt},X_{\mt}).$$
\end{proof}

Using Lemma \ref{vetgeo} one can show that if $X$ is a geodesic vector and $\gamma$ is the corresponding homogeneous geodesic, then $E^\prime(0)=0$ for any proper variation $f$. In fact, 
\begin{eqnarray*}
E^\prime(0) &=& \int_0^a \! B_\Lambda((\Ad(\exp{-tX})q'(t))_{\mt},X) \, \dn t \\
& = & \int_0^a \! B_\Lambda(q'(t)_{\mt},X) \, \dn t\\
& = & \left. B_\Lambda(q(t)_{\mt},X) \right\vert_{0}^a\\
& = & B_\Lambda(q(a)_{\mt},X) - B_\Lambda(q(0)_{\mt},X)\\
& = & B_\Lambda((\Ad{(\exp{-aX})}q(a))_{\mt},X) - B_\Lambda(q(0)_{\mt},X)\\
& = & 0,
\end{eqnarray*}
since $q(0)$ and $Ad(\exp{-aX})q(a) \in \kt$.
\\

We are now able to compute the second variation of the energy of an homogeneous geodesic $\gamma$.

\begin{proposition} \label{sve}
Let $X$ be a geodesic vector with respect to an invariant metric $\Lambda$. The {\em second variation of the energy} of the homogeneous geodesic $\gamma$ is given by 
\begin{eqnarray*}
\sve = E''(0) & = & 2\int_0^a \! B_\Lambda([q(t),q'(t)]_{\mt},X) \, \dn t\\
& & + 2\int_0^a \! B_\Lambda((\Ad{(\exp{-tX})}q'(t))_{\mt},(\Ad{(\exp{-tX})}q'(t))_{\mt}) \, \dn t
\end{eqnarray*}
\end{proposition}

\begin{proof} 
We have
\begin{eqnarray*}E'(s) = E'(\gamma^s) & = & \int_0^a \! B_\Lambda \left( \dfrac{\partial}{\partial s}(\Ad{(\exp{-tX})}C_s)_{\mt}, (\Ad(\exp{-tX})C_s)_{\mt} + 2X \right) \, \dn t\\
& & + \int_0^a \! B_\Lambda\left((\Ad(\exp{-tX})C_s)_{\mt},\dfrac{\partial}{\partial s}(\Ad{(\exp{-tX})}C_s)_{\mt}\right) \, \dn t,
\end{eqnarray*}

and

\begin{eqnarray*}E''(0) & = & 2\int_0^a \! B_\Lambda \left( \left. \dfrac{\partial^2}{\partial s^2}(\Ad{(\exp{-tX})}C_s)_{\mt} \right\vert_{s=0},X \right) \, \dn t\\
&  & + 2 \int_0^a \! B_\Lambda((\Ad(\exp{-tX})q'(t))_{\mt},(\Ad(\exp{-tX})q'(t))_{\mt}).
\end{eqnarray*}

Now, observe that

\begin{eqnarray*}
& &\left.\dfrac{\partial^2}{\partial s^2}\Ad{(\exp{-tX})}C_s\right\vert_{s=0}\\
& & = \dfrac{\partial^2}{\partial s^2}\left.\left(\sum_{n \geq 0}\dfrac{t^n\ad(X)^n}{n!}sq'(t) + \sum_{n \geq 0}\dfrac{t^n\ad(X)^n}{n!}\dfrac{s^2[q(t),q'(t)]}{2} + O(s^3)\right)\right\vert_{s=0}\\
&& = \Ad{(\exp{-tX})[q(t),q'(t)]}.
\end{eqnarray*}

Therefore,

\begin{eqnarray*}
E''(0) & = &  2\int_0^a \! B_\Lambda((\Ad(\exp{-tX})[q(t),q'(t)])_{\mt},X) \, \dn t\\
&&
+ 2\int_0^a \! B_\Lambda((\Ad{(\exp{-tX})}q'(t))_{\mt},(\Ad{(\exp{-tX})}q'(t))_{\mt}) \, \dn t\\
& = & 2\int_0^a \! B_\Lambda([q(t),q'(t)]_{\mt},X) \, \dn t\\
&&
+ 2\int_0^a \! B_\Lambda((\Ad{(\exp{-tX})}q'(t))_{\mt},(\Ad{(\exp{-tX})}q'(t))_{\mt}) \, \dn t,
\end{eqnarray*}
since $X$ is a geodesic vector with respect to the invariant metric $\Lambda$.
\end{proof}

\section{Perturbation lemma for homogeneous geodesics.}\label{pertub}

Consider $\Lambda$ a fixed invariant metric on $\mathbb{F}_\Theta$. In this section we will provide a formula to compute the second variation of the energy with respect to a $\Pp$-perturbed metric $\Lambda^\#$ of $\Lambda$ in terms of $I^{\gamma}_\Lambda$ . An similar result was proved in \cite{caiolinosan} in the context of harmonic maps.

Let $\Lambda$ be an invariant metric on $\flag$ defined by $\Lambda = (\lambda_\sigma)_{\sigma \in \piteta}$ and $\Pp$ be a subset of $\piteta$. 
\begin{definition}
The invariant metric $\Lambda^\# = (\lambda_\sigma^\#)_{\sigma \in \piteta}$ is called a {\em $\Pp$-perturbation of $\Lambda$} if the following holds:
\begin{enumerate}
\item $\lambda_\sigma^\# = \lambda_\sigma$ for all $\sigma \in \Pp$;
\item $\lambda_\sigma^\# = \lambda_\sigma + \xi_\sigma > 0$, $\xi_\sigma \in \R$,  for any $\sigma \in \piteta\setminus\Pp$. 
\end{enumerate}
\end{definition}

Let $\gamma: I \to \flag$ be an homogeneous curve given by $\gamma(t) = \exp tX \cdot o$ with $X \in \mt$. Since $U$ acts on $\flag$ by isometries the set $\{\Aal,\Sal\}_{\alpha \in \pimenos}$ is a basis of $\mt \simeq T_o\flag$. One can see that $$\{(d \phi_{\exp{tX}})_o(\Aal),(d \phi_{\exp{tX}})_o(\Sal)\}_{\alpha \in \pimenos}$$ is an orthogonal frame along $\gamma(t)$. Then, there exists differentiable maps  $f_\alpha, g_\alpha: I \to \R$, $\alpha \in \pimenos$ such that
$$\gamma'(t) = \sum_{\alpha \in \pimenos}(f_\alpha(t)(d \phi_{\exp{tX}})_o(\Aal) + g_\alpha(t)(d \phi_{\exp{tX}})_o(\Sal)).$$

We define $\gamma_\alpha(t) := f_\alpha(t)(\dn a_{\exp{tX}})_o(\Aal) + g_\alpha(t)(\dn a_{\exp{tX}})_o(\Sal)$. This motivates the following definition:

\begin{definition}
An homogeneous curve $\gamma: I \rightarrow \flag$ given by $\gamma(t)=\exp tX \cdot o$, $X \in \mt$, is {\em subordinated } to  $\Pp$ if $\gamma_\alpha = 0$ for all $\alpha \in \sigma \in \piteta\setminus\Pp$.
\end{definition}

In other words, we say that $\gamma$ is subordinated to $\mathcal{P}$ if the image of $\psi$ is tangent to the distribution spanned by $\mathcal{P}$.
\\

Now we face the following situation: suppose that a (homogeneous) curve is a geodesic with respect to the invariant metric $\Lambda$ and is also a geodesic with respect to another metric $\Lamjv$ (obtained by a $\mathcal{P}$-perturbation of $\Lambda$). What is the relation between the second variation of the energy of $\gamma$ with respect to both metrics? The answer for this question is our next result, called {\em the perturbation lemma}.


\begin{lemma}
Let $\gamma: [0,a] \rightarrow \flag$ be an homogeneous curve subordinated to $\Pp$ and let $\Lamjv = (\lamjv)$ be a $\Pp$-perturbation of $\Lambda = (\lambda_\sigma)$. Assume that  $\gamma$ is an homogeneous geodesics with respect to $\Lambda$ and $\Lamjv$ with geodesic vector $X \in \mt$. Then 
$$\svejv = \sve + 4\sum_{\alpha \in \sigma}\sum_{\sigma \in \piteta\setminus\Pp}\xi_\sigma \int_0^a(\tilde{f}_\alpha^2(t) + \tilde{g}_\alpha^2(t))dt,$$
where $\tilde{f}_\alpha(t)$ e $\tilde{g}_\alpha(t)$ are differentiable functions given by 
$$(\Ad{(\exp{-tX})q'(t))_{\mt}} = \sum_{\alpha \in \pimenos}(\tilde{f}_\alpha(t)\Aal + \tilde{g}_\alpha(t)\Sal).$$
\end{lemma}
\begin{proof}
We have
$$[q(t),q'(t)]_{\mt} = \sum_{\alpha \in \pimenos}(\hat{f}_{\alpha}(t)\Aal + \hat{g}_\alpha(t)\Sal)$$
and
$$X = \sum_{\alpha \in \pimenos} (a_\alpha \Aal
+ b_\alpha \Sal),$$
where $\hat{f}_\alpha(t)$ and $\hat{g}_\alpha(t)$ are differentiable functions and $a_\alpha$, $b_\alpha$ $\in \R$. Then,
\begin{eqnarray*}
\svejv & = & 2\int_0^a \! B_{\Lamjv}((\Ad{(\exp{-tX})}q'(t))_{\mt},(\Ad{(\exp{-tX})}q'(t))_{\mt}) \, \dn t\\
&& + 2\int_0^a \! B_{\Lamjv}([q(t),q'(t)]_{\mt},X) \, \dn t\\
           & = & 4\sum_{\alpha \in \sigma}\sum_{\sigma \in \piteta}\lamjv \int_0^a(\tilde{f}_\alpha^2(t) + \tilde{g}_\alpha^2(t))dt + 4\sum_{\alpha \in \sigma}\sum_{\sigma \in \piteta}\lamjv \int_0^a(a_\alpha \hat{f}_\alpha(t) + b_\alpha \hat{g}_\alpha(t))dt. 
\end{eqnarray*}

If $\alpha \in \sigma \in \piteta \setminus \Pp$ then $a_\alpha = 0 = b_\alpha$, and 
\begin{eqnarray}
\svejv & = & 4\sum_{\alpha \in \sigma}\sum_{\sigma \in \piteta\setminus \Pp}\lamjv \int_0^a(\tilde{f}_\alpha^2(t) + \tilde{g}_\alpha^2(t))dt + 4\sum_{\alpha \in \sigma}\sum_{\sigma \in \Pp}\lamjv \int_0^a(\tilde{f}_\alpha^2(t) + \tilde{g}_\alpha^2(t))dt\nonumber \\
          &   & + 4\sum_{\alpha \in \sigma}\sum_{\sigma \in \Pp}\lamjv \int_0^a(a_\alpha \hat{f}_\alpha(t) + b_\alpha \hat{g}_\alpha(t))dt. \nonumber \\
          & = & 4\sum_{\alpha \in \sigma}\sum_{\sigma \in \piteta\setminus \Pp}(\lambda_\sigma + \xi_\sigma) \int_0^a(\tilde{f}_\alpha^2(t) + \tilde{g}_\alpha^2(t))dt + 4\sum_{\alpha \in \sigma}\sum_{\sigma \in \Pp}\lambda_\sigma \int_0^a(\tilde{f}_\alpha^2(t) + \tilde{g}_\alpha^2(t))dt\nonumber \\
          &   & + 4\sum_{\alpha \in \sigma}\sum_{\sigma \in \Pp}\lambda_\sigma \int_0^a(a_\alpha \hat{f}_\alpha(t) + b_\alpha \hat{g}_\alpha(t))dt. \nonumber \\
          & = & 4\sum_{\alpha \in \sigma}\sum_{\sigma \in \piteta\setminus \Pp}(\lambda_\sigma + \xi_\sigma) \int_0^a(\tilde{f}_\alpha^2(t) + \tilde{g}_\alpha^2(t))dt + 4\sum_{\alpha \in \sigma}\sum_{\sigma \in \Pp}\lambda_\sigma \int_0^a(\tilde{f}_\alpha^2(t) + \tilde{g}_\alpha^2(t))dt\nonumber \\
          &   & + 4\sum_{\alpha \in \sigma}\sum_{\sigma \in \piteta}\lambda_\sigma \int_0^a(a_\alpha \hat{f}_\alpha(t) + b_\alpha \hat{g}_\alpha(t))dt. \nonumber \\
          & = & \sve + 4\sum_{\alpha \in \sigma}\sum_{\sigma \in \piteta\setminus \Pp}\xi_\sigma \int_0^a(\tilde{f}_\alpha^2(t) + \tilde{g}_\alpha^2(t))dt. \nonumber
\end{eqnarray}
\end{proof}
\begin{remark}
The hypothesis in the perturbation lemma that $\gamma$ is geodesic with respect to both metrics $\Lambda$ and $\Lamjv$ is essential in the computations. This assumption is fulfilled by the {\em homogeneous equigeodesics}, that is, curves that are geodesics with respect to any invariant metrics.
\end{remark}

\section{Homogeneous Ricci flow and conjugate points on $\C P^{2n+1}=\Spp(n+1) / \text{U}(1)\times \Spp(n)$.} 
\subsection{Review on the geometry of  $\C P^{2n+1}$.}
Let us review some well known facts about the geometry of $\Spp(n+1)$-invariant metrics on the complex projective space $\C P^{2n+1}$. Let us denote by $\spp(n+1)$ the Lie algebra of the Lie groups $\Spp(n+1)$ and by $\kt$ the Lie algebra of the isotropy subgroup $\text{U}(1)\times \Spp(n)$ of the $\Spp(n+1)$-action on $\C P^{2n+1}$. We also consider the reductive decomposition of $\spp(n+1)$, that is, 
\begin{equation}
\spp(n+1)=\kt \oplus \mt,
\end{equation} 
where $\ad(\kt)\mt \subset \mt$, and we identify $\mt$ with the tangent space at the origin (trivial coset) of $\C P^{2n+1}$.

The isotropy representation splits into two irreducible components, namely
\begin{equation}
\mt=\m_1\oplus \m_2,
\end{equation}
with $\dim \m_1=4n$ and $\dim \m_2=2$. The decomposition $\spp(n+1) = \kt \oplus \m_1 \oplus \m_2$ satisfies the following Lie brackets relations:

\begin{equation}
[\m_1,\m_2] \subset \m_1, \qquad [\m_1,\m_1] \subset \m_2 \oplus \kt, \qquad [\m_2,\m_2] \subset \kt.
\label{rel}
\end{equation}

We define the subalgebra $\h = \kt \oplus \m_2$ and by $(\ref{rel})$ we have
\begin{equation}
[\h,\h] \subset \h, \qquad [\h,\m_1] \subset \m_1, \qquad [\m_1,\m_1] \subset \h.
\end{equation}

Therefore $(\mathfrak{g},\mathfrak{h})$ is a symmetric pair and the corresponding symmetric space is $\Spp(n+1)/(\Spp(1)\times\Spp(n))$, the Grassmannian of quaternionic 1-dimensional of $\mathbb{H}^{n+1}$.

Since $\text{U}(1)\times\Spp(n) \subset \Spp(1)\times\Spp(n) \subset \Spp(n+1)$ we have the following twistor fibration (cf. \cite{B-R})
\begin{equation}\label{twistor}
\Spp(1)/\text{U}(1)  \cdots \Spp(n+1)/\text{U}(1)\times\Spp(n)  \to \Spp(n+1)/ \Spp(1)\times\Spp(n),
\end{equation}
or equivalently $S^2\cdots  \C P^{2n+1} \to \mathbb{H}P^{n}$.

One can describe precisely each one of the irreducible components $\m_1, \m_2$ in terms of the root space decomposition of the Lie algebra $\spp(n+1)$.  We will use the following basis of $\spp(n+1)$: let $E_{ij}$ $(1 \leq i,j \leq 2n+2)$  be the $(2n+2) \times (2n+2)$ matrix with $1$ in the $ij$-position and zero otherwise. Let $m = \sqrt{2n+4}/(2n+4)$ and define


\begin{eqnarray*}
X_{ij} & = & \dfrac{\sqrt{2}m}{2}(E_{ij}-E_{n+1+j,n+1+i}), \\
X_{-ij} & = & \dfrac{\sqrt{2}m}{2}(E_{ji} - E_{n+1+i,n+1+j}), \\
X_{ij}^{+} & = & \dfrac{\sqrt{2}m}{2}(E_{i,n+1+j}+E_{j,n+1+i}), \\
X_{-ij}^{+} & = &\dfrac{\sqrt{2}m}{2}(E_{n+1+i,j}+E_{n+1+j,i}),
\end{eqnarray*}
$1 \leq i < j \leq n+1$,
\begin{eqnarray*}
X_{ii} & = & m(E_{i,n+1+i}) \\
X_{-ii} & = & m(E_{n+1+i,i}),
\end{eqnarray*}
$1 \leq i \leq n+1$.

The matrices defined above are the Weyl basis of $\spp(2n+2,\C)$ associated to the Cartan subalgebra of diagonal matrices. Let us use the compact real form of $\spp(2n+2,\C)$ . We define

$$A_{ij} = X_{ij} - X_{-ij} \text{, } S_{ij} = i(X_{ij}+X_{-ij}),$$
$1 \leq i \leq j \leq n+1$,
$$A_{ij}^{+} = X_{ij}^{+} - X_{-ij}^{+} \text{, } S_{ij}^{+} = i(X_{ij}^{+} + X_{-ij}^{+}),$$
$1 \leq i < j \leq n+1$, and
$$H_{i,i+1} = [X_{i,i+1},X_{-(i,i+1)}] \text{, } H_{n+1,n+1} = [X_{n+1,n+1},X_{-(n+1,n+1)}],$$
$1 \leq i \leq n$. 

Therefore, a basis of  $\spp(n+1)$ is given by the vectors 
$$\{A_{ii},S_{ii},A_{n+1,n+1},S_{n+1,n+1},A_{ij},S_{ij},A_{ij}^{+},S_{ij}^{+},iH_{i,i+1},iH_{n+1,n+1} \, ; \, 1 \leq i < j \leq n+1\}.$$

The decomposition $\spp(n+1) = \kt \oplus \mt$ is given by

\begin{equation}
\kt  =  \text{sp}_{\R}\{A_{ii},S_{ii},A_{n+1,n+1},S_{n+1,n+1},A_{ij},S_{ij},A_{ij}^{+},S_{ij}^{+},iH_{i,i+1},iH_{n+1,n+1},iH_{12} \, ; \, 2 \leq i<j \leq n+1\} 
\end{equation}
and
\begin{equation}\label{basem}
 \mt=\text{sp}_{\R} \{A_{11},S_{11},A_{1j},S_{1j},A_{1j}^{+},S_{1j}^{+} \, ; \, 2 \leq j \leq n+1\},
\end{equation}

and the irreducible components of the isotropy representation $\mt = \m_1 \oplus \m_2$ are
\begin{eqnarray*}
\m_1 & = & \Sp\{A_{1j},S_{1j},A_{1j}^{+},S_{1j}^{+}\, ; \, 2 \leq j \leq n+1\}, \\
\m_2 & = & \Sp\{A_{11},S_{11}\}.
\end{eqnarray*}

\begin{example} Let us consider the simplest case $n=1$. The basis of $\mathfrak{sp}(2)$ as defined above is 
\begin{center}
$A_{11} = \left( \begin{array}{cccc}
0 & 0 & \frac{\sqrt{6}}{6} & 0 \\
0 & 0 & 0 & 0 \\
-\frac{\sqrt{6}}{6} & 0 & 0 & 0 \\
0 & 0 & 0 & 0
\end{array} \right), \
S_{11} = \left( \begin{array}{cccc}
0 & 0 & \frac{\sqrt{6}i}{6} & 0 \\
0 & 0 & 0 & 0 \\
\frac{\sqrt{6} i}{6} & 0 & 0 & 0 \\
0 & 0 & 0 & 0
\end{array} \right), \
\newline
A_{22} = \left( \begin{array}{cccc}
0 & 0 & 0 & 0 \\
0 & 0 & 0 & \frac{\sqrt{6}}{6} \\
0 & 0 & 0 & 0 \\
0 & -\frac{\sqrt{6}}{6} & 0 & 0
\end{array} \right), \
S_{22} = \left( \begin{array}{cccc}
0 & 0 & 0 & 0 \\
0 & 0 & 0 & \frac{\sqrt{6}i}{6} \\
0 & 0 & 0 & 0 \\
0 & \frac{\sqrt{6} i}{6} & 0 & 0
\end{array} \right), \
\newline
A_{12} = \left( \begin{array}{cccc}
0 & \frac{\sqrt{3}}{6} & 0 & 0 \\
-\frac{\sqrt{3}}{6} & 0 & 0 & 0 \\
0 & 0 & 0 & \frac{\sqrt{3}}{6} \\
0 & 0 & -\frac{\sqrt{3}}{6} & 0
\end{array} \right), \
S_{12} = \left( \begin{array}{cccc}
0 & \frac{\sqrt{3} i}{6} & 0 & 0 \\
\frac{\sqrt{3} i}{6} & 0 & 0 & 0 \\
0 & 0 & 0 & -\frac{\sqrt{3} i}{6} \\
0 & 0 & -\frac{\sqrt{3} i}{6} & 0
\end{array} \right), \
\newline
A_{12}^+ = \left( \begin{array}{cccc}
0 & 0 & 0 & \frac{\sqrt{3}}{6} \\
0 & 0 & \frac{\sqrt{3}}{6} & 0 \\
0 & -\frac{\sqrt{3}}{6} & 0 & 0 \\
-\frac{\sqrt{3}}{6} & 0 & 0 & 0
\end{array} \right), \
S_{12}^+ = \left( \begin{array}{cccc}
0 & 0 & 0 & \frac{\sqrt{3} i}{6} \\
0 & 0 & \frac{\sqrt{3} i}{6} & 0 \\
0 & \frac{\sqrt{3} i}{6} & 0 & 0 \\
\frac{\sqrt{3} i}{6} & 0 & 0 & 0
\end{array} \right),\
\newline
iH_{12} = \left( \begin{array}{cccc}
\frac{i}{12} & 0 & 0 & 0 \\
0 & -\frac{i}{12} & 0 & 0 \\
0 & 0 & -\frac{i}{12} & 0 \\
0 & 0 & 0 & \frac{i}{12}
\end{array} \right),\
iH_{22} = \left( \begin{array}{cccc}
0 & 0 & 0 & 0 \\
0 & \frac{i}{6} & 0 & 0 \\
0 & 0 & 0 & 0 \\
0 & 0 & 0 & -\frac{i}{6}
\end{array} \right).
$
\end{center}

In this case the projective space is $\C P^3 = \Spp(2) / \normalfont{\text{U}}(1)\times\Spp(1) $, the Lie algebra of the isotropy group is 
$$\kt = \Sp \{A_{22},S_{22},iH_{12},iH_{22}\},$$
the tangent space at the origin is the vector space
$$\mt = \Sp \{A_{11},S_{11},A_{12},S_{12},A_{12}^+,S_{12}^+\},$$
and the irreducible components of the isotropy representation are
\begin{eqnarray*}
\m_1 & = & \Sp \{A_{12},S_{12},A_{12}^+,S_{12}^+\} \text{ and} \\
\m_2 & = & \Sp \{A_{11},S_{11}\}.
\end{eqnarray*}
The homogeneous fibration defined in $(\ref{twistor})$ is the Penrose's twistor fibration $\mathbb{C}P^3\to \mathbb{H}P^1 $.
\end{example}

\begin{remark}
The basis of the Lie algebra $\mathfrak{sp}(n)$ described above is motivated by the {\em Weyl's trick} for compact real forms of complex simple Lie algebras and this basis is different from the {\em quaternionic} basis for $\mathfrak{sp}(n)$. 
\end{remark}

Since we have two isotropy summands, each $\Spp(n+1)$-invariant metric on $\C P^{2n+1}$ is determined by two positive real numbers: $\Lambda=(\lambda_1,\lambda_2)$. The normal homogeneous {\em standard}  Riemannian metric is given by $\Lambda_{NR}=(1,1)$. This metric is clearly naturally reductive.

The classification of $\Spp(n+1)$-invariant Einstein metrics on $\C P^{2n+1}$ was provided by Ziller.
\begin{theorem}{\cite{ziller}}
The projective space $\C P^{2n+1}$ admits two $\Spp(n+1)$-invariant Einstein metrics (up to scale).
\end{theorem}


The {\em cut locus} of a naturally reducible homogeneous space was also described by Ziller in \cite{ziller2} (see also \cite{priconj}).  In the case of $\C P^{2n+1}$ equipped with the normal metric one can describe their cut locus as follow (cf. \cite{priconj}, Thm 2.9): consider the projective space $\C P^{2n+1} = \text{Sp}(n+1)/(\text{U}(1)\times\text{Sp}(n))$ equipped with
$\Lambda_{NR}=(1,1)$, the normal invariant metric, and let $\gamma: \R \to
\C P^{2n+1}$ be an homogeneous geodesic given by $\exp{tX} \cdot o$, where
$X$ is an element of the basis of $\mt$ defined in (\ref{basem}). Then,
the first conjugate point of $\gamma(0)$ is $\gamma(\pi/2m)$, where
$m=\sqrt{2n+4}/(2n+4)$.

%




Now we will use the Perturbation Lemma in order to prove the first main result of this section, that asserts that one can produces conjugate points along an equigeodesic by deforming the normal metric on $\C P^{2n+1}$ via a $\mathcal{P}$-perturbation. If we choose an equigeodesic vector on the irreducible component $\mathfrak{m}_2$, then a $\mathcal{P}$-perturbation of an invariant metric $\Lambda=(\lambda_1,\lambda_2)$ is given in the $\mathfrak{m}_1$-direction (that is, perturbing the $\lambda_1$-parameter of $\Lambda$) according to the following result: 

\begin{theorem}\label{conj-p-pertub}
Consider the projective space $\C P^{2n+1} = \text{Sp}(n+1)/(\text{U}(1)\times\text{Sp}(n))$ equipped with $\Lambda_{NR}=(1,1)$, the normal invariant metric, and let $\gamma: \R \to \C P^{2n+1}$ be the equigeodesic 
\begin{equation}
\gamma(t) = \exp{tA_{11}} \cdot o,
\end{equation}
where $A_{11}\in \mathfrak{m}_2\subset\mathfrak{sp}(n+1)$. Fix $b\in \R$ such that $0 < b < \pi/2m$, where \linebreak $m=\sqrt{2n+4}/(2n+4)$. Then there exist an invariant metric $\Lamjv$ obtained by a $\mathcal{P}$-perturbation of the metric $\Lambda_{NR}$ and $c \in \R$, $0 < c \leq b$, such that $\gamma(0)$ and $\gamma(c)$ are conjugate with respect to the metric $\Lamjv$.
\end{theorem}

\begin{proof}
Remember that $\gamma(\pi/2m)$ is the first conjugate point of $\gamma(0)$ along $\gamma$, with respect to the metric $\Lambda_{NR}$ (see \cite{priconj}).

Consider the interval $[0,b]$. Then, 
\begin{equation}
\sveNR \geq 0,
\end{equation}
for any differentiable variation $q: [0,b] \to \spp(n+1)$ such that
\begin{equation}
q(0), \, \Ad(\exp -tA_{11})q(b) \in \kt.
\end{equation}

Define $q_0: [0,b] \to \mathfrak{sp}(n+1)$ by
$$q_0(t) = k\sen\left(\dfrac{2\pi t}{b}\right)A_{12} + \dfrac{1}{k}t(t-b)A_{12}^+,$$
where $k$ is a non zero constant. The derivative of $q$ is 
$$q_0'(t) = \left(\dfrac{2\pi k}{b}\cos{\left(\dfrac{2\pi t}{b}\right)}\right)A_{12} + \dfrac{1}{k}(2t-b)A_{12}^{+}.$$
Thus
\begin{eqnarray*}
(\Ad{(\exp{-tA_{11}})}q_0'(t))_{\mt} & = & \left(\dfrac{2\pi k}{b}\cos{\left(\dfrac{2\pi t}{b}\right)}\cos{(mt)}-\dfrac{1}{k}(2t-b)\sen{(mt)}\right)A_{12}\\
&& + \left(\dfrac{2\pi k}{b}\cos{\left(\dfrac{2\pi t}{b}\right)}\sen{(mt)}+ \dfrac{1}{k}(2t-b)\cos{(mt)}\right)A_{12}^+
\end{eqnarray*}
and
$$[q_0(t),q_0'(t)]_{\mt} = m\left((2t-b)\sen{\left(\dfrac{2\pi t}{b}\right)} - \dfrac{2\pi t(t-b)}{b}\cos{\left(\dfrac{2\pi t}{b}\right)}\right)A_{11}.$$

Let $M =I_{\Lambda_{NR}}^\gamma(q_0)$ defined on the interval $[0,b]$. By Proposition \ref{sve}, 
\begin{eqnarray*}
M & = & 4m\int_{0}^{b} \left((2t-b)\sen{\left(\dfrac{2\pi t}{b}\right)} - \dfrac{2\pi t(t-b)}{b}\cos{\left(\dfrac{2\pi t}{b}\right)}\right) \, dt \\
& & + \, 4\int_{0}^{b} \left\{\left(\dfrac{2\pi k}{b}\cos{\left(\dfrac{2\pi t}{b}\right)}\right)^2 + \left(\dfrac{1}{k}(2t-b)\right)^2\right\} \, dt.
\end{eqnarray*}

Consider
$$N = \int_{0}^{b} \left\{\left(\dfrac{2\pi k}{b}\cos{\left(\dfrac{2\pi t}{b}\right)}\right)^2 + \left(\dfrac{1}{k}(2t-b)\right)^2\right\} \, dt.$$
Then,
$$ N = \dfrac{b^4 + 6\pi^2k^4}{3k^2b} > 0$$
and
$$ M - 4N = -\dfrac{8mb^2}{\pi} < 0.$$

Now, let $\xi \in \R$ such that $-1 < \xi < -M/4N$ and consider a $\mathcal{P}$-perturbation of $\Lambda_{NR}$ given by $\Lamjv = (1+\xi,1)$. By the Perturbation Lemma,
\begin{equation}
\svejvv = M + 4\xi N < M + 4\left(-\dfrac{M}{4N}\right)N = 0.
\end{equation}
Therefore, there exists $c \in \R$, $0 < c \leq b$ such that $\gamma(c)$ is conjugated to $\gamma(0)$ with respect to the metric $\Lamjv$.
\end{proof}

\begin{remark}
The computations above work in the same way for $A_{1j}$ and  $A_{1j}^+$, where $1 < j \leq n+1$, instead of $A_{12}$ and $A_{12}^+$. 
We also remark that the analysis for the equigeodesic vector $S_{11}$ is similar, considering $S_{1j}^+$ instead $A_{1j}^+$.
\end{remark}

\subsection{Homogeneous Ricci Flow}

Let us review shortly the global behavior of the homogeneous Ricci flow on $\C P^{2n+1} = \Spp(n+1) / \text{U}(1)\times\Spp(n)$. For more details, see \cite{gramamartins}.
We keep our notation and denote an invariant metric just by $\Lambda=(\lambda_1,\lambda_2)$.  The Ricci tensor of an invariant metric $\Lambda$ is again an invariant tensor, and therefore completely determined by its value at the origin of the homogeneous space and constant on each irreducible component of the isotropy representation. In the case of  $\C P^{2n+1}$, the components of the Ricci tensor are given by
\begin{equation}
\label{77tfu}
\left\{
\begin{array}{rcl}
\vspace{0.5cm}\displaystyle r_1&=&\displaystyle -\frac{4}{2n+4}-\left(\frac{2n}{4n+8}\right)\frac{\lambda_1^2}{\lambda_2^2},\\
\displaystyle r_2&=&\displaystyle -\frac{7+4n}{4n+8}+\left(\frac{6}{16n+32}\right)\frac{ \lambda_1}{ \lambda_2}.
\end{array}
\right.
\end{equation}

The Ricci flow equation on the manifold $M$ is defined by
\begin{equation}
\label{eqnricci}
\frac{\partial g(t)}{\partial t}=-2Ric(g(t)),
\end{equation}
where $Ric(g)$ is the Ricci tensor of the Riemannian metric $g$. The solution of this equation, the so called Ricci flow, is a $1$-parameter family of metrics $g(t)$ in $M$.

In our example, the Ricci flow equation for invariant metrics is given by the following ODE system: 
\begin{equation}\label{ricci-eqn}
\begin{cases} \dot{x} = \dfrac{4}{2n+4} + \left(\dfrac{2n}{4n+8}\right)\dfrac{x^2}{y^2}, \\
\\
\dot{y} = \dfrac{7 + 4n}{4n + 8} - \left(\dfrac{6}{16n+32}\right)\dfrac{x}{y}.
\end{cases}
\end{equation}

The invariant lines of the system $(\ref{ricci-eqn})$ are
\begin{equation}
\gamma_1(t) = \left(t,\left(\dfrac{4n+3}{8}\right)t\right)
\end{equation}
and
\begin{equation}
\gamma_2(t)=\left(t,\dfrac{1}{2}t\right).
\end{equation}
The global behavior of the Ricci flow on $\C P^{2n+1}$ is described using its phase portrait described in the Figure \ref{retratofase}. 
\begin{figure}[!htb]
\centering
\includegraphics[scale=1]{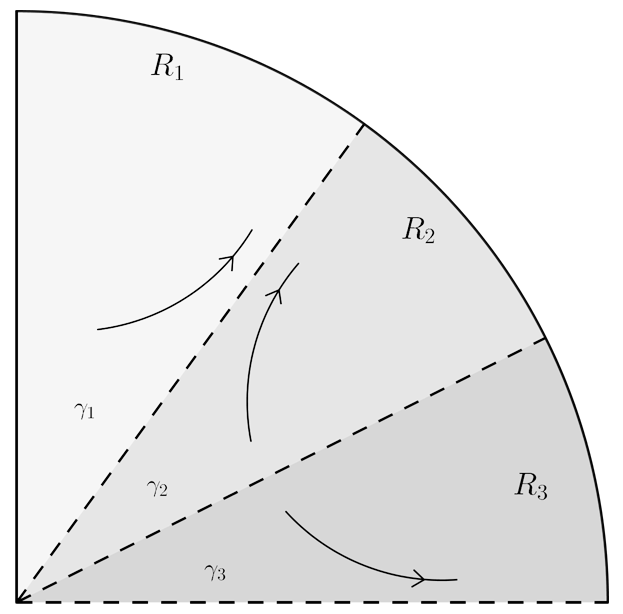}
\caption{Phase portrait of the Ricci flow on $\C P^{2n+1}$}
\label{retratofase}
\end{figure}

%
%
One can describe precisely the ``asymptotic behavior'' of the flows line of the Ricci flow. Let $g_0$ be an invariant metric and $g(t)$ the Ricci flow with initial condition $g_0$. We will denote by $g_\infty$ the limit $\lim_{t\to -\infty} g(t)$. 
\begin{theorem}\label{teoricci}{\cite{gramamartins}} 
Let $g_0$ be an invariant metric on $\C P^{2n+1}$ and $R_1$, $R_2$, $R_3$, $\gamma_1$ and $\gamma_2$ as described in Figure \ref{retratofase}. We have:
\begin{enumerate}
\item[a)] if $g_0 \in R_1\cup R_2 \cup \gamma_1$ then $g_{\infty}$ is the Einsten (non-K\"ahler) metric;
\item[b)] if $g_0 \in \gamma_2$ then $g_{\infty}$ is the Kähler-Einstein metric.
\item[c)] if $g_0 \in R_3$, we consider the twistor fibration $S^2\cdots \C P^{2n+1} \to \mathbb{H}P^n$. Then the Ricci flow $g(t)$ with $g(0) = g_0$ evolves in such a way that the diameter of fiber converges to zero when $t \to -\infty$.
\end{enumerate}
\end{theorem}

From now on we consider the normal metric $g_0 = \Lambda_{NR}=(1,1) \in R_2$ on $\C P^{2n+1}$ and $\gamma$ an homogeneous equigeodesic such that $\gamma(0)$ is the origin of the homogeneous space. According to the previous section, the first conjugate point to $\gamma(0)$ is $\gamma(\frac{\pi}{2m})$. Now we will deal with the following question: can the Ricci flow $g(t)$ with initial condition $g_0$ evolve with the property that given $0<b<\frac{\pi}{2m}$, there exist $t_0 \in (-\infty,0)$ and $c \in \R$, $0 < c \leq b$, such  that $\gamma(c)$ and $\gamma(0)$ are conjugate with respect to the invariant metric $g(t_0)$? In other words, can we produce conjugate points deforming a metric via the Ricci flow ?

We will give a positive answer to this question for $\C P^{2n+1}$, with $n\geq 10$. 

Let $0<b<\frac{\pi}{2m}$ and consider a $\mathcal{P}$-perturbation of the initial metric $g_0 = (1,1)$ given by $\Lamjv = (\lambjv_1,1)$ with $$\lambjv_1 \in (0,1-M/4N),$$where  the $M,N$ are computed in the proof of Theorem $\ref{conj-p-pertub}$ (and its depends of $b$). 

Our first remark is that if $n \geq 10$, the invariant metric $\Lamjv$ can be taken in the region $R_2$ of the phase portrait of the Ricci flow system. Since the invariant line $\gamma_1$ of the phase portrait is generated by the invariant Einstein (non-K\"ahler) metric it is sufficient to show that $1-M/4N>8/(4n+3)$.



\begin{lemma}\label{lemak} Consider  $\C P^{2n+1} = \text{Sp}(n+1)/(\text{U}(1)\times\text{Sp}(n))$, with $n \geq 10$, and let $\gamma(t) = \exp tX \cdot o$ be an equigeodesic with $X = A_{11} \in \m_1$. If $f(s,t)$ is a geodesic variation of $\gamma$ given by 
$$f(s,t) = \exp sq(t) \cdot \gamma(t),$$
with $q: [0,b] \to \mathfrak{sp}(n+1)$ defined by 
$$q(t) = k\sen{\left(\dfrac{2\pi t}{b}\right)}A_{12} + \dfrac{1}{k}t(t-b)A_{12}^+$$
and $b \in \left(\dfrac{8\sqrt{6}\pi^2}{3m(4n+3)},\dfrac{\pi}{2m}\right)$, $m=\sqrt{2n+4}/(2n+4)$, then there exists $k$ such that 
$$1 - \dfrac{M}{4N} > \dfrac{8}{4n + 3}.$$
\end{lemma}
\begin{proof}
Using the computations of the Theorem $\ref{conj-p-pertub}$ we have
$$\dfrac{4N-M}{4N} = \dfrac{2b^2m}{\pi}\dfrac{3k^2b}{b^4+6\pi^2k^4} = \dfrac{6k^2b^3m}{\pi(b^4+6\pi^2k^4)}.$$

Therefore $(4N-M)/4N$ is a function in the variable $k$ and its maximum value is $bm\sqrt{6}/2\pi^2$ at $k = \pm b/\sqrt[4]{6}\sqrt{\pi}$. Since $b>\dfrac{8\sqrt{6}\pi^2}{3m(4n+3)}$ we have $$\dfrac{bm\sqrt{6}}{2\pi^2} > \dfrac{8}{4n+3}.$$

In order to conclude the proof, we consider $k = b/\sqrt[4]{6}\sqrt{\pi}$ or $k = -b/\sqrt[4]{6}\sqrt{\pi}$.  

\end{proof}

\begin{remark} In the Lemma \ref{lemak} we need to assume  $n \geq 10$ in order to guarantee that $8\sqrt{6}\pi^2/3m(4n+3) < \pi/2m$.  Furthermore, on the choice of $q(t)$, given  $j$ such that  $1<j\leq n+1$, one can use  $A_{1j}$ and $A_{1j}^+$ instead $A_{12}$ and $A_{12}^+$, respectively.
\end{remark}

\begin{remark} If we consider the equigeodesic vector $X=S_{11} \in \m_1$ a similar analysis hold; we just consider $S_{1j}^+$ instead $A_{1j}^+$, $1 < j \leq n+1$.
\end{remark}
Now we are able to prove the main theorem of this section.
\begin{theorem}
Consider  $\C P^{2n+1} = \text{Sp}(n+1)/(\text{U}(1)\times\text{Sp}(n))$, with $n \geq 10$,  and $$b \in \left(\dfrac{8\sqrt{6}\pi^2}{3m(4n+3)},\dfrac{\pi}{2m}\right),$$
where $m=\sqrt{2n+4}/(2n+4)$. Let $\gamma(t)=\exp tX \cdot o$ be an equigeodesic with \linebreak $X = A_{11} \in \m_1$.   Then the homogeneous Ricci flow $g(t)$ with $g(0)=\Lambda_{NR} =(1,1)$ evolves in such way that there exist $t_0\in (-\infty,0)$ and $c \in \R$, $0 < c \leq b$, such that $\gamma(0)$ and $\gamma(b)$ are conjugate with respect to the invariant metric $g(t_0)$.

\end{theorem}
\begin{proof}
Initially, we take $\zeta$ in the interval  $(8/(4n+3),1-M/4N)$. By Theorem \ref{conj-p-pertub} and Lemma \ref{lemak} the invariant metric determined by $(\zeta,1)$ is in the region $R_2$ for the phase portrait of the Ricci flow and $\gamma(b)$ is conjugate to $\gamma(0)$ with respect to this metric.

Now consider homoteties of the metric  $(\zeta,1)$ given by the line
\begin{equation}
\ell (r)=r\cdot (\zeta,1)=(r\zeta,r), \,\,\, r>0.
\end{equation}
The entire line $\ell(r)$ is contained in the region $R_2$. Since $\disp\lim_{t \to \infty}{g_t}$ converges asymptotically to the line $\gamma_1$, the Intermediate Value Theorem guarantees the existence of $r_0$ and $t_0$ such that $$g(t_0)=\ell(r_0)=(r_0\,\zeta,r_0). $$
For the metric $g(t_0)$ we have $I_{{g(t_0) }}^{\gamma}(q) = r_0 I_{(\zeta,1)}^{\gamma}(q) $ with 
$$q(t) = \dfrac{b}{\sqrt[4]{6}\sqrt{\pi}}\sen{\left(\dfrac{2\pi t}{b}\right)}A_{12} + \dfrac{\sqrt[4]{6}\sqrt{\pi}}{b}t(t-b)A_{12}^+.$$
Since $I_{(\zeta,1)}^{\gamma}(q) <0$ and $r_0>0$ we have $ I_{{g(t_0) }}^{\gamma}(q)<0$. This means that there exists $0 < c \leq b$ such that $\gamma(c)$ and $\gamma(0)$ are conjugate with respect to the metric $g(t_0)$.
\end{proof}



\begin{example}
The following picture illustrate the flow line of the Ricci flow with initial metric $g(0)=(1,1)$, the perturbed metric $(\xi,1)$ and the line of homoteties $\ell(r)$ on $\C P^{21}$ (that is, $n=10$).
\begin{figure}[!htb]
\centering
\includegraphics[scale=1]{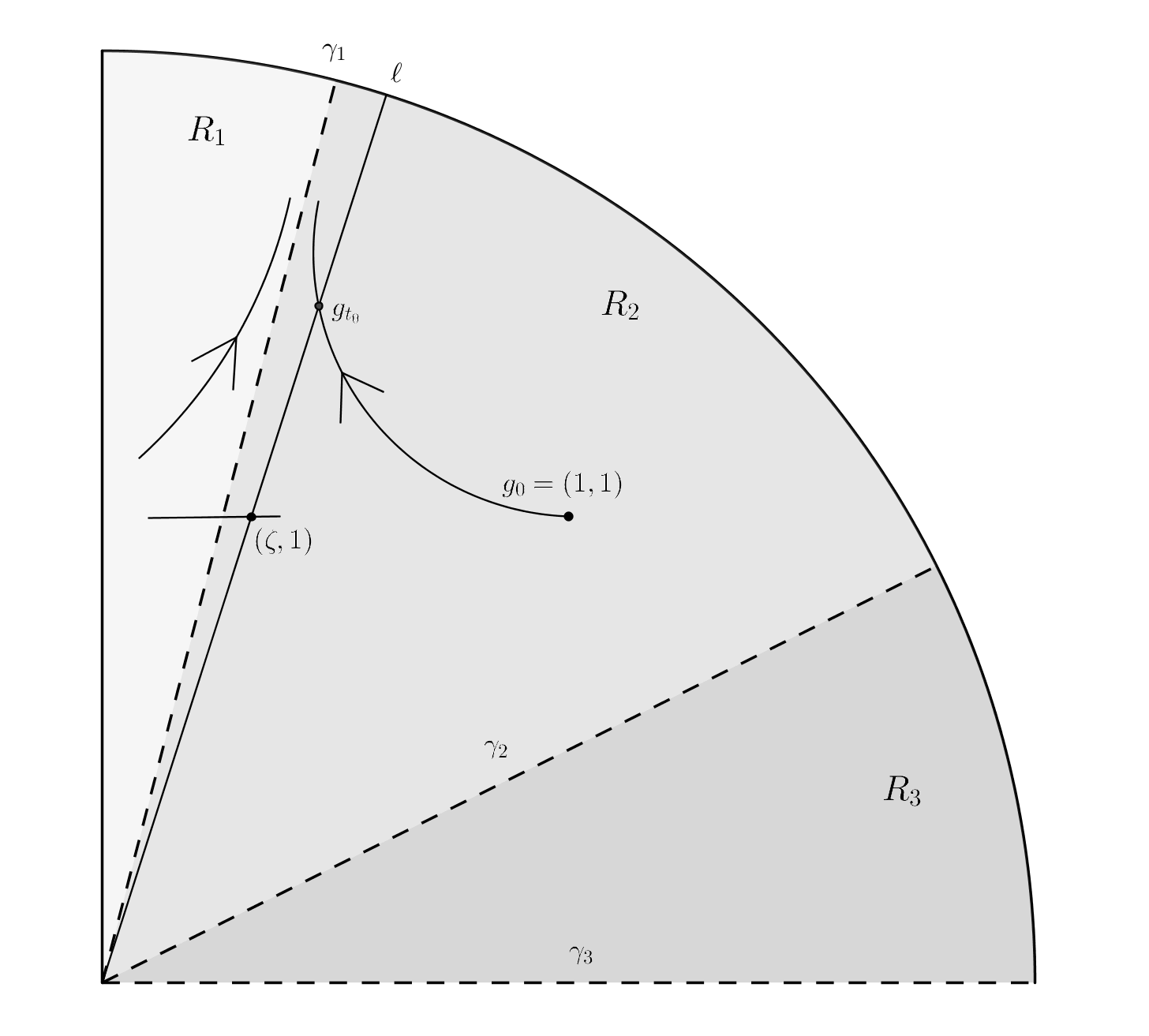}
\label{exemplocp15}
\end{figure}

\end{example}

\section{Algebraic Criteria}

In this section we provide a general procedure to deform the normal metric (parametrized by $\Lambda=(1,1)$) in a generalized flag manifold in order to produce a conjugate point along a homogeneous geodesics. This is done by producing a specific  $\mathcal{P}$-perturbation, and one can characterize this perturbation using  Lie theoretical description of generalized flag manifolds.   

\begin{definition} \label{pair-pertub}
Let $\alpha \in \sigma_o \in \piteta$. A pair of roots $\{\beta,\gamma\}$ is called a {\em perturbation pair with respect to }$\alpha$ if its satisfies the following conditions:
\begin{enumerate}
\item $\beta$, $\delta \in \pimenos$,
\item $\beta$, $\delta \notin \sigma_0$,
\item $\alpha = \beta + \delta$,
\item $\alpha + \beta$ and $\alpha + \delta \notin \Pi$,
\item $\beta - \delta \notin \pimenos$.
\end{enumerate}
\end{definition}

The next Lemma is very well known and we will use it in the proof of the main result of this section.
\begin{lemma}\label{somadosm}
\normalfont{\textbf{(\cite{hel}, p. 171)}} Suppose $\alpha$, $\beta$, $\gamma \in \Pi$ and $\alpha + \beta + \gamma = 0$. Then,
$$m_{\alpha,\beta} = m_{\beta,\delta} = m_{\delta,\alpha}.$$
\end{lemma}

Now we are able to state the main result of this section. This is a generalization of the procedure to produce conjugate points in $\mathbb{CP}^{2n+1}$.

\begin{theorem}\label{ca}
Let $\flag = G/\Pt = U/\Kt$ be a generalized flag manifold equipped with the standard metric $\Lambda$. Let $\alpha \in \sigma_o \in \piteta$ and  consider the equigeodesic $\gamma: \R \to \flag$ given by $$\gamma(t) = \exp{tA_{\alpha}}\cdot o.$$ Suppose that there exists $a \in \R$ such that $\gamma(a)$ is the first conjugate point of $\gamma(0)$ along $\gamma$ with respect to the metric $\Lambda$. If $\{\beta,\gamma\}$ is a perturbation pair with respect to $\alpha$, then there exist a differentiable map $$q_0: [0,b] \to \ul$$ and an invariant metric $\Lamjv$ obtained by a $\Pp$-perturbation of $\Lambda$ such that $$\svejvv < 0,$$ where $0<b<a$.
\end{theorem}

\begin{proof}
We have that $\gamma(a)$ is the first conjugate point of $\gamma(0)$ along $\gamma$. Then, considering the interval $[0,b]$, $\sve \geq 0$ for all differentiable map $q: [0,b] \to \ul$ such that $q(0)$ and $\Ad{(\exp{-t\Aal})}q(b) \in \kt$.

By Lemmas \ref{colchetes} and \ref{somadosm}
\begin{eqnarray*}
[\Aal,\Abeta] & = & m_{-\alpha,\beta}\Adelta \\
{[\Aal,\Adelta]} & = & -m_{-\alpha,\beta}\Abeta \\
{[\Abeta,\Adelta]} & = & m_{-\alpha,\beta}\Aal + m_{-\beta,\delta}A_{\beta - \delta}.
\end{eqnarray*}

We will split our proof in two cases:
\begin{itemize}
\item \textit{Case 1}: $m_{-\alpha,\beta} > 0$
\end{itemize}
Define $q_0: [0,b] \to \ul$ by
$$q_0(t) = k\sen\left(\dfrac{2\pi t}{b}\right)\Abeta + \dfrac{1}{k}t(t-b)\Adelta,$$
where $k$ is a non zero constant. The derivative of $q$ is
$$q_0'(t) = \left(\dfrac{2\pi k}{b}\cos{\left(\dfrac{2\pi t}{b}\right)}\right)\Abeta + \dfrac{1}{k}(2t-b)\Adelta.$$
Thus,
\begin{eqnarray*}
(\Ad{(\exp{-t\Aal})}q_0'(t))_{\mt} & = & \left(\dfrac{2\pi k}{b}\cos{\left(\dfrac{2\pi t}{b}\right)}\cos{(m_{-\alpha,\beta}t)}-\dfrac{1}{k}(2t-b)\sen{(m_{-\alpha,\beta}t)}\right)\Abeta\\
&& + \left(\dfrac{2\pi k}{b}\cos{\left(\dfrac{2\pi t}{b}\right)}\sen{(m_{-\alpha,\beta}t)}+ \dfrac{1}{k}(2t-b)\cos{(m_{-\alpha,\beta}t)}\right)\Adelta
\end{eqnarray*}
and
$$[q_0(t),q_0'(t)]_{\mt} = m_{-\alpha,\beta}\left((2t-b)\sen{\left(\dfrac{2\pi t}{b}\right)} - \dfrac{2\pi t(t-b)}{b}\cos{\left(\dfrac{2\pi t}{b}\right)}\right)\Aal.$$

Let $M = \svee$ on the interval $[0,b]$. Then, according to Proposition \ref{sve},
\begin{eqnarray*}
M & = & 4m_{-\alpha,\beta}\int_{0}^{b} \left((2t-b)\sen{\left(\dfrac{2\pi t}{b}\right)} - \dfrac{2\pi t(t-b)}{b}\cos{\left(\dfrac{2\pi t}{b}\right)}\right) \, dt \\
& & + \, 4\int_{0}^{b} \left\{\left(\dfrac{2\pi k}{b}\cos{\left(\dfrac{2\pi t}{b}\right)}\right)^2 + \left(\dfrac{1}{k}(2t-b)\right)^2\right\} \, dt.
\end{eqnarray*}

Let
$$N = \int_{0}^{b} \left\{\left(\dfrac{2\pi k}{b}\cos{\left(\dfrac{2\pi t}{b}\right)}\right)^2 + \left(\dfrac{1}{k}(2t-b)\right)^2\right\} \, dt.$$
Then,
$$ N = \dfrac{b^4 + 6\pi^2k^4}{3k^2b} > 0$$
and
$$ M - 4N = -\dfrac{8m_{-\alpha,\beta}b^2}{\pi} < 0.$$

We consider $\sigma_1,\sigma_2 \in \piteta$, $\sigma_0 \neq \sigma_1, \sigma_2$, such that $\beta \in \sigma_1$ e $\delta \in \sigma_2$ ($\sigma_1$ and $\sigma_2$ can be the same). Let $\Pp = \{\sigma_0\}$ and consider a $\Pp$-perturbation $\Lamjv = (\lamjv)_{\sigma \in \piteta}$ of $\Lambda$ given by:
\begin{enumerate}
\item $\lambda_{\sigma_0}^\# = 1$;
\item $\lambda_{\sigma_1}^\# = 1 + \xi$;
\item $\lambda_{\sigma_2}^\# = 1 + \xi$;
\item $\lamjv = 1 + \xi_{\alpha} > 0$, $\xi_{\alpha} \in \R$, for all $\sigma \in \piteta\setminus\{\sigma_0,\sigma_1,\sigma_2\}$,
\end{enumerate}
where $1 < \xi < -M/4N$. Then, by the Perturbation Lemma,
$$\svejvv = M + 4\xi N < M + 4\left(-\dfrac{M}{4N}\right)N = 0.$$

Therefore, there exists a constant $c$, $0<c\leq b$, such that $\gamma(c)$ is conjugated to $\gamma(0)$ with respect to the metric $\Lamjv.$
\begin{itemize}
\item \textit{Case 2:} $m_{-\alpha,\beta} < 0$
\end{itemize}

In this case, we define $q_0: [0,b] \to \ul$ by
$$q_0(t) = k\sen\left(\dfrac{2\pi t}{b}\right)\Adelta + \dfrac{1}{k}t(t-b)\Abeta.$$

The analysis is analogous to the previous case and we conclude that
$$M - 4N = \dfrac{8m_{-\alpha,\beta}b^2}{\pi} < 0.$$
Then, it is enough to consider the same $\Pp$-perturbation.
\end{proof}

\begin{example}
Consider the Wallach flag manifold $SU(3)/T^2$ equipped with the standard metric $(1,1,1)$.

The complexification of $\su(3)$ is the Lie algebra $\sll(3,\C)$. The root space decomposition of $\sll(3,\C)$ is given as follows. Consider the Cartan subalgebra $\h$ given by diagonal matrices of trace zero. Then, the root system of $\sll(3,\C)$ relative to $\h$ is composed by $\alpha_{ij} := \varepsilon_i - \varepsilon_j$, $1 \leq i \neq j \leq 3$, where $\varepsilon_i$ is the functional given by $\varepsilon_i: diag\{a_1,a_2,a_3\} \to a_i$.

A simple system of roots is
$$\Sigma = \{\alpha_{12}, \alpha_{23}\}$$
and the positive roots relative to this simple system are
$$\Pi^+ = \{\alpha_{12},\alpha_{13},\alpha_{23}\}.$$

The Weyl basis of $\sll(n,\C)$ associated to the Cartan subalgebra $\h$ is composed by
$$X_{ij} = X_{\alpha_{ij}} = \dfrac{\sqrt{2n}E_{ij}}{2n}$$
and
$$H_{i,i+1} = H_{\alpha_{i,i+1}} = \dfrac{E_{ii}-E_{i+1,i+1}}{2n},$$
$1 \leq i \neq j \leq 3$, where $E_{ij}$ is the $n \times n$ matrix with 1 in the $ij$-position and zero otherwise. Then, a basis for the real compact form $\su(n)$ of $\sll(n,\C)$ is given by the vectors $A_{ij} = X_{ij}-X_{ji}$, $S_{ij} = \sqrt{-1}(X_{ij} + X_{ji})$ and $\sqrt{-1}H_{i,i+1}$, $1 \leq i < j \leq 3$.

In the case of the flag manifold $SU(3)/T^2$, we have that $\Theta = \emptyset$ and $\ul_{12}$, $\ul_{13}$ and $\ul_{23}$ are the irreducible components of the isotropy representation, where
$$\ul_{ij} = \Sp\{A_{ij},S_{ij}; 1\leq i < j \leq 3\},$$
i.e.,
$$\mt = \ul_{12} \oplus \ul_{13} \oplus \ul_{23}.$$

Now consider the equigeodesic $\gamma: \R \to \ul$ given by $\gamma(t) = \exp{tA_{13}}\cdot o$. We have $\alpha_{13} = \alpha_{12} + \alpha_{23}$. Notice that $\alpha_{12} - \alpha_{23}$, $\alpha_{13}+\alpha_{12}$, $\alpha_{13} + \alpha_{23} \notin \Pi$. 

The first conjugate point of $\gamma(0)$ is $\gamma(\sqrt{6}\pi/2)$. Fix $b \in \R$ such that $0<b<\sqrt{6}\pi/2$. Then, by Theorem \ref{ca}, there exist a differentiable map $q_0: [0,b] \to \su(3)$ and an invariant metric $\Lamjv$ given by a $\Pp$-perturbation of $\Lambda$ such that $\svejvv < 0$. In fact, we just define $q_0$ according to
$$q_0(t) = k\sen\left(\dfrac{2\pi t}{b}\right)A_{12} + \dfrac{1}{k}t(t-b)A_{23}$$
and consider the $\{\alpha_{13}\}$-perturbation $\Lamjv = (\lambda_{12}^\#,\lambda_{13}^\#,\lambda_{23}^\#)$ given by:
\begin{enumerate}
\item $\lambda_{12}^\# = \lambda_{23}^\# = 1 + \xi$;
\item $\lambda_{13}^\# = 1$,
\end{enumerate}
where 
$$-1 < \xi < \dfrac{\sqrt{6}b^3k^2 - \pi b^4 - 6k^4\pi^3}{\pi(b^4+6k^4\pi^2)}.$$ \qed

\end{example}


\begin{thebibliography}{99}


\bibitem{niko} N.A. Abiev, Yu.G. Nikonorov;  The evolution of positively curved invariant Riemannian metrics on the Wallach spaces under the Ricci flow. arXiv:1509.09263 

\bibitem{A-A}D. Alekseevsky, A.Arvanitoyeorgos; Riemannian flag manifolds with homogeneous geodesics, {\it Trans. Amer. Math. Soc.} {\bf 359} (2007), 3769--3789.

\bibitem{gk1} A. L. Besse; \textit{Einstein Manifolds}. Springer, 2008.

\bibitem{B-W} C.Bohm and B.Wilking; Nonnegatively curved manifolds with finite fundamental groups admit metrics with
positive Ricci curvature, {\it GAFA, Geom. Func. Anal.} (2007),{\bf 17}, 665--681.

\bibitem{B-R} F.Burstal, J.Rawsley; Twistor Theory for Riemannian Symmetric Spaces, Lecture Notes in Mathematics {\bf 1424} 1990.

\bibitem{chavel} I. Chavel; Isotropic Jacobi fields and Jacobi's equations on Riemannian homogeneous spaces. \textit{Comm. Math. Helv.} \textbf{42} (1967), 237 - 248.

\bibitem{cheeger-ebin} J.Cheeger, D. Ebin, Comparison Theorems in Riemannian Geometry. AMS Chelsea Publishing (1975).

\bibitem{lino1} N. Cohen, L. Grama, C. J. C. Negreiros; Equigeodesics on flag manifolds. \textit{Houston Math. Journal} \textbf{37} No.1 (2011), 113-125.

\bibitem{priconj} J. C. González-Dávila, A. M. Naveira; Existence of non-isotropic conjugate points on rank one normal homogeneous spaces. \textit{Ann. Glob. Anal. Geom.} \textbf{45} (2014), 211 - 231.

\bibitem{gramamartins} L. Grama, R. M. Martins; Global behavior of the Ricci flow on generalized flag manifolds with two isotropy summands. \textit{Indag. Math.} \textbf{23}, (2012) 95-104.

\bibitem{hel} S. Helgason; \textit{Diferential geometry and symmetric spaces}. Academic Press, New York-
London, 1962.

\bibitem{itoh} M. Itoh; On curvature properties of K\"ahler C-spaces. \textit{J. Math. Soc. Japan}, \textbf{30} No.1 (1978), 39-71.

\bibitem{koba} S. Kobayashi, K. Nomizu; \textit{Foundations of Differential Geometry}, Vol. II. Interscience, Wiley, Nova Iorque, 1969. 

\bibitem{kosz} O. Kowalski, J. Szenthe; On the existence of homogeneous geodesics in homogeneous Riemannian manifolds. \textit{Geo. Dedicata} \textbf{81} (2000), 209-214.

\bibitem{kow} O. Kowalski, L. Vanhecke; Riemannian manifolds with homogeneous geodesics. \textit{Bolletino U.M.I} \textbf{7(5-B)} (1991), 189-246.



\bibitem{lauret2} R.Lafuente and J.Lauret   Structure of homogeneous Ricci solitons and the Alekseevskii conjecture, {\it Journal of Differential Geometry} {\bf 98} (2014), 315--347.


\bibitem{lauret3} J.Lauret; Curvature flows for almost hermitian Lie groups, to appear in  {\it Trans. Amer. Math. Soc.}.


\bibitem{caiolinosan} C. J. C. Negreiros, L. Grama, L. A. B. San Martin; Invariant Hermitian structures and variational aspects of a family of holomorphic curves on flag manifolds. \textit{Ann. Glob. Anal. Geom.} \textbf{40} (2011), 105 - 123.

\bibitem{sm-negr}L. A. B. San Martin and C. J. C. Negreiros; Invariant almost Hermitian structures on flag manifolds, \textit{Advances in Mathematics} 178 (2003), 277--310.

\bibitem{sm-alg}L. A. B. San Martin; \'Algebras de Lie, Editora Unicamp (2003).



\bibitem{ziller} W. Ziller; Homogeneous Einstein Metrics on Spheres and Projective Spaces. \textit{Math. Ann.} \textbf{259} (1982), 351 - 358.

\bibitem{ziller2} W. Ziller; The jacobi equation on naturally reductive compact Riemannian homogeneous spaces, Comm. Math. Helvetici, 52 (1977), 573--590. 

\end{thebibliography}
\end{document}